\documentclass[a4paper]{article}
\usepackage[english]{babel}

\usepackage{latexsym}

\usepackage{enumerate}

\usepackage{graphicx}
\usepackage{graphics}
\usepackage{subfigure}
\usepackage{amsmath}
\usepackage{amsfonts}
\usepackage{amssymb}
\usepackage{amscd}
\usepackage{amsthm}

\usepackage{multirow}
\usepackage{fancyvrb}
\usepackage{color}

\usepackage{verbatim}

\usepackage{hyperref}
\hypersetup{colorlinks=false}

\usepackage{pgfplots}
\tikzset{ font={\fontsize{9pt}{12}\selectfont}}

\newtheorem{theorem}{Theorem}[section]
\newtheorem{proposition}[theorem]{Proposition}
\newtheorem{lemma}[theorem]{Lemma}

\newtheorem{definition}[theorem]{Definition}
\newtheorem{remark}[theorem]{Remark}

\newcommand{\C}{\mathbb{C}}
\newcommand{\WC}{\widehat{\mathbb{C}}}
\newcommand{\R}{\mathbb{R}}

\newcommand{\N}{\mathbb{N}}

\newcommand{\com}{\mathbb{C}}
\newcommand{\wcom}{\widehat{\mathbb{C}}}

\begin{document}

\title{Dynamics of Newton-like root finding methods }
\author{B. Campos, J. Canela and \footnote{Corresponding author} P. Vindel \\
   \footnotesize Instituto de Matem\'aticas y Aplicaciones de Castell\'on, \\
  \footnotesize    Universitat Jaume I, Castell\'on, Spain \\
  \footnotesize  campos@uji.es,  canela@uji.es, vindel@uji.es}

\date{}
\maketitle

\begin{abstract}

When exploring the literature, it can be observed that the operator obtained when applying \textit{Newton-like} root finding algorithms to the quadratic polynomials $z^2-c$ has the same form regardless of which algorithm has been used.  In this paper we justify why this expression is obtained.  
This is done by studying the symmetries of the operators obtained after applying Newton-like algorithms to a family of  degree $d$ polynomials $p(z)=z^d-c$.

 Moreover,  we provide an iterative procedure  to obtain the expression of new Newton-like algoritms.

We also carry out a dynamical study of the given generic  operator and provide general conclusions of this type of methods.
\end{abstract}

\vspace{1cm} 

\vspace{0.5cm}

\section{Introduction}

Numerical methods allow finding solutions of non-linear equations that cannot be solved by algebraic procedures.  The development and improvement of these methods and their behaviour constitute a field of intense research and a vast literature related to this topic can be found; see for example \cite{mcnamee2013numerical}, \cite{ostrowski2016solution},   \cite{PETKOVIC2013},  \cite{petkovic},  \cite{traub1982},  and references therein.

Iterative methods solve non-linear equations by generating successive approximations that may eventually converge to the solution. The so-called one-point-methods (see \cite{traub1982}) start with an initial guess of the solution to proceed iteratively so that each approximation is used to obtain the next one, until a desired level of convergence is reached. That is, an approximation of the solution $x^*$ of an equation $f(x)=0 $ can be found by applying an iterative scheme of the form:
\[
x_{k+1}=\phi (x_k), \ \ k=0,1,2,...
\]
where $x_k$ is an approximation of $x^*$ and $\phi$ is the iteration operator. The function $\phi$ may depend on derivatives of $f$ in order to increase the order of convergence.

The simplest and most popular method using this scheme is the well-known Newton’s method, also known as Newton–Raphson's method (see \cite{Blanchard94}), given by:
\[
x_{k+1}=N_f(x_k):=x_k-\frac{f(x_k)}{f'(x_k)}.
\]
To improve the efficiency of one-point methods, multi-point methods were introduced.  For them,  each step is not exclusively based on the previous iteration, but also includes information of intermediate evaluations.  These  schemes, mainly variants of Newton’s method,  have recently gained interest because they provide root-finding algorithms which improve both convergence order and computational efficiency in comparison with one-point ones. Such advantages allow  to optimize computational resources. However, the radii of convergence
which ensure that the solution of the method is correct  decrease when the order of the method increases (see \cite{radioCov}, for example).

One way to assess the balance between benefits and drawbacks is to study the dynamical behaviour of these  methods. By considering a discrete dynamical system and performing a qualitative study of it, we can obtain dynamical spaces that identify the regions where the method displays good behaviour. Moreover, if the method has parameter dependence constituting a family of methods, we can plot the corresponding parameter spaces  to find the members of the family with better behaviour (see \cite{CCV-gradon},  \cite{CTV-gato}, for example).

The first works in this sense (  \cite{Kneisl}, for example) show the dynamical planes of the best known numerical  methods acting on polynomials of degree two or three. In \cite{CTV-gato} the authors began the dynamical study of  the one-parameter family called Chebyshev-Halley. 

In the present paper we study multipoint Newton-like methods, that is, methods where the intermediate evaluations are variations of Newton's scheme.

As can be observed in the literature (see \cite{AM-pg5},  \cite{CCV-gradon+k},  \cite{CCTV-familiac}, \cite{CV-OstrowskiChun},  \cite{CGTVV-king},  \cite{CGMT-pg4},  \cite{CTV-gato},   \cite{MA-pg2},   \cite{ZCT-pg8}, for example), when  Newton-like algorithms are applied to the quadratic polynomials $p(z)=z^2-c$, the operator obtained is conjugated (via a M\"obius map that sends the $\pm\sqrt{c}$ to 0 and $\infty$) with a rational map which has the following generic expression:
\begin{equation}\label{operador}
	O(z)=z^n \frac{a_k+a_{k-1} z+...+a_{1} z^{k-1}+ z^k}{1+a_{1} z+...+a_{k-1} z^{k-1}+a_k z^k}.
\end{equation}

The main goal of this paper is to justify why this kind of operator is obtained regardless of the Newton-like algorithm used. This is done in Theorem \ref{teorema}. This theorem follows from a more general symmetry property of the iterative schemes obtained from Newton-like algorithms applied to the degree $d$ polynomials $p(z)=z^d-c$ (Theorem~\ref{prop:symn}). The symmetry was noted by Chun et al. \cite{CHUN2012}, but they didn’t draw conclusions.

The idea of how to justify that operator \eqref{operador} is obtained when applying Newton-like algorithms to $p(z)=z^2-c$ is the following. In Theorem~\ref{prop:symn} we describe a recurrent procedure which can be used to describe the operators obtained from Newton-like root-finding algorithms when applied to the degree $d$ polynomials  $p(z)=z^d-c$. Moreover, we prove that the maps obtained fromTheorem~\ref{prop:symn} are symmetric with respect to rotation by a $d$th root of the unity. By restricting to $d=2$ and applying a  conjugacy, we conclude (Theorem \ref{teorema}) that the operators obtained for the different families of  Newton-like methods when they are applied on quadratic polynomials are given by Equation~(\ref{operador}).

In \S~\ref{clasificacion} we carry out a dynamical study of the operator (\ref{operador}) in order to obtain generic conclusions of this type of methods. In Proposition~\ref{0y1} we show that the points $z=0$, $z=\infty $ and $z=1$ are always fixed points of this operator and $z=-1$ is also a fixed point if $n+k$ is odd. The points $z=0$ and $z=\infty$ correspond to the zeros of the quadratic polynomial, which are superattracting fixed points of local degree $n$ (the numerical methods have order of convergence $n$ to the roots). In order to establish the existence of stable behaviour other than the basins of attraction of the roots, we also study the sets of parameters where $z=1$ and $z=-1$ are attracting fixed points (Propositions \ref{estabilidad_1} and \ref{estabilidad_-1}). This study is done under the extra hypothesis that the relation between the coefficients of the rational function and the parameter of the family is linear, which is a common phenomenon in the literature.

Section \S~\ref{casos} is devoted to display some known  examples of Newton-like algorithms that correspond to families with one free critical point, which allows us to draw the parameter planes of the families. The operators of this section follow Equation (\ref{operador}). It is also observed that, under given conditions, two different methods whose operators,  expressed as (\ref{operador}), have the same values of $n$ and $k$ present similar dynamics ( \S~\ref{n>k+1}).
Moreover, in \S~\ref{n=k}, we study the particular case in which $n=k$ and all the coefficients of the rational function, except $a_k$, are real. 

Finally, we want to emphasize that the functions introduced in Theorem~\ref{prop:symn} can be used to generate new Newton-type algorithms for solving nonlinear equations.
In fact, in Section \S~\ref{simetria} we show that the algorithms of the most usual methods, such as  Newton, Traub, Halley, Chebyshev, and Jarratt, are obtained  from the functions introduced in Theorem~\ref{prop:symn}. Although we do not intend to carry out an exhaustive verification, it is easy to extend this assertion to other Newton-like algorithms. 

In the construction of new Newton-type algorithms, parameters must be introduced in the iterative procedure applied to functions given in Theorem~\ref{prop:symn}. These parameters must be adjusted so that the method obtained has the desired order of convergence.

We also want to emphasize that the main goal of this paper is to study general properties of Newton-like algorithms and to make a qualitative study of them. The construction of methods using the iterative procedure provided by Theorem~\ref{prop:symn} is the aim of a coming paper.

For the sake of completeness, in \S~\ref{construccion} we present some of the best known Newton-like methods of different degree and in \S~\ref{sec:introhol} we recall some basic concepts of complex dynamics.

 \subsection{Newton-like methods}\label{construccion}

As mentioned before, multi-point iteration methods were introduced to improve the order of convergence and the efficiency of iterative methods and they are described by means of the expressions $w_1(x_k),w_2(x_k),...,w_n(x_k)$.  The iteration function $\phi$, defined as
\[
x_{k+1}=\phi \left(x_k,w_1(x_k),w_2(x_k),...,w_n(x_k)\right)
\]
is called a multi-point iteration function without memory.

The simplest examples are \emph{Steffensen's method}, with $w_1(x_k)=x_k+f(x_k)$:
\[
x_{k+1}=x_k-\frac{ f(x_k)^2}{f(x_k+f(x_k))-f(x_k)}
\]
and \emph{Traub-Steffensen's method}, with $w_1(x_k)=x_k+\gamma f(x_k)$:
\[
x_{k+1}=S(x_k):=x_k-\frac{\gamma f(x_k)^2}{f(x_k+\gamma f(x_k))-f(x_k)}.
\]
There are many  multi-point methods and we do not pretend here to do an exhaustive study of them; we  only point out those that we consider best known. A more comprehensive study can be seen in  \cite{ostrowski2016solution}, \cite{PETKOVIC2013} and \cite{traub1982},  for example.

Following the summary provided in \cite{petkovic}, different two-step methods can be built by using  Newton's method as pre-conditioner, such as \emph{Traub's scheme} \cite{traub1982}:
\begin{eqnarray*}
  y_k &=& x_k-\frac{f(x_k)}{f'(x_k)} \\
  x_{k+1} &=&  y_k - \frac{f(y_k)}{f'(x_k)}
\end{eqnarray*}
or
\emph{Ostrowski's scheme} \cite{ostrowski2016solution}:
\begin{eqnarray*}
y_k &=& x_k- \frac{f(x_k)}{f'(x_k)} \\
  x_{k+1} &=&  y_k-\frac{f(y_k)}{f'(x_k)} \frac{f(x_k)}{f(x_k)-2f(y_k)}.
\end{eqnarray*}
The latest was generalized by King \cite{king}, who defined the scheme:
\begin{eqnarray} \label{fking}
	\notag
y_k &=& x_k- \frac{f(x_k)}{f'(x_k)} \\
  x_{k+1} &=&  y_k-\frac{f(y_k)}{f'(x_k)} \frac{f(x_k)+\beta f(y_k)}{f(x_k)+(\beta-2)f(y_k)},
\end{eqnarray}
obtaining a family of iterative methods depending on one parameter.  It was shown by Chun et al. \cite{CHUN2012} that the best parameter is $\beta=0.$

If instead of using the Newton's method as the first step, the so-called Jarratt method step is used, a method of order of convergence four is obtained. The resulting scheme is called Jarratt method  \cite{jarratt}:
\begin{eqnarray}\label{ecJarrat}
\notag y_n &=& x_n- \frac{2}{3} \frac{f(x_n)}{f'(x_n)} \\
x_{n+1} &=&  x_n- J_f(x_n) \frac{f(x_n)}{f'(x_n)}
\end{eqnarray}
where:
\[
 J_f(x_n) =  \frac{3f'(y_n)+f'(x_n)}{2\left(3f'(y_n)-f'(x_n)\right)}.
 \]

From the two-step methods, other three-step methods can be obtained, with general scheme:
\begin{eqnarray*}
  y_k &=& x_k-\frac{f(x_k)}{f'(x_k)} \\
  z_k &=& \phi (x_k,y_k) \\
  x_{k+1} &=&  z_k-\frac{f(z_k)}{f'(z_k)}.
\end{eqnarray*}
Some examples of schemes obtained using this procedure are the sixth-order convergence method deduced by Wang et al.\ \cite{wang},
\begin{eqnarray*}
  x_{n+1} &=&x_n-J_f(x_n) \frac{f(x_n)}{f'(x_n)}\\ 
  y_n &=& x_n- \frac{2}{3} \frac{f(x_n)}{f'(x_n)} \\
J_f(x_n) &=&  \frac{3f'(y_n)+f'(x_n)}{6f'(y_n)-2f'(x_n)},
\end{eqnarray*}
the fourth-order family of iterative method introduced by Amat et al.\ in \cite{amat2005},

\begin{eqnarray} \label{famat}
	\notag
u_f(z)&=&  \frac{f(x_n)}{f'(x_n)} \\
\notag
 h_f(z) &=&  \frac{f'(z-\frac{2}{3}u_f(z))-f'(z)}{f'(z)}\\
J(z) &=&  z - u_f(z)+\frac{3}{4}u_f(z)h_f(z) \frac{1+\beta h_f(z)}{1+(\frac{3}{2}+\beta)h_f(z)},
\end{eqnarray}
and the family of sixth-order methods defined by Chun \cite{chun},
\begin{eqnarray*}
 y_n &=& x_n- \frac{2}{3} \frac{f(x_n)}{f'(x_n)} \\
  z_n &=&x_n-J_f(x_n) \frac{f(x_n)}{f'(x_n)}\\
x_{n+1} &=&  z_n - \frac{f(z_n)}{\alpha (z_n-x_n)(z_n-y_n)+\frac{3}{2}J_f(x_n) f'(y_n)+\left(1-\frac{3}{2}J_f(x_n)\right)f'(x_n)}.
\end{eqnarray*}

One of the best known Newton-like method corresponds to  the one-parameter family called Chebyshev-Halley.
\begin{eqnarray} \label{fcheby}
	\notag
	y_k &=& x_k-\frac{f(x_k)}{f'(x_k)} \\
	\notag
	L_f(x_k) &=& \frac{f(x_k) f''(x_k)}{(f'(x_k))^2}\\
	x_{k+1} &=&  y_k-\frac{1}{2}\frac{L_f(x_k)}{1-\alpha L_f(x_k)}\frac{f(x_k)}{f'(x_k)}.
\end{eqnarray}

This family includes the Chebyshev's method when the parameter $\alpha$  is equal to $0$, Halley's scheme for $\alpha = 1/2$,  and Newton’s method when $\alpha$ tends to $ \infty$. In \cite{CTV-gato} the authors began the dynamical study of this family applied on arbitrary polynomials of degree two. 

This type of dynamical study has been extended to other families of numerical methods, such as the King family \cite{CGTVV-king}, the $c$-family \cite{CCTV-familiac}, etc.
Nowadays, there is a wide literature expanding this study to methods with higher order of convergence (see for example \cite{AM-pg5}, \cite{CHUN2012},     \cite{junjua}, and references therein). In these papers,  the dynamical behaviour of these families applied on quadratic polynomials is considered. Dynamical studies of methods, or families of numerical methods, applied on polynomials of higher degree can be found (see  \cite{CCV-gradon+k}, \cite{CCV-gradon}, \cite{Gutierrezcubicos}, for example).

\subsection{Basic concepts of complex dynamics}\label{sec:introhol}

Before starting the dynamical study of these operators, we briefly recall the basic concepts of complex (as opposed to real) dynamics that we use in this paper. For a more detailed introduction to the topic of complex dynamics we refer to \cite{Milnor}.

Given a
rational map $R:\wcom\rightarrow \wcom$, where $\wcom$
denotes the Riemann sphere,  we  consider the dynamical system given by the iterates of $R$. We say that a point $z_0\in\wcom$ is \textit{fixed} if $R(z_0)=z_0$  and \textit{periodic} of period $p$ if $R^p(z_0)=z_0$, where $p$ is minimal. The multiplier of a fixed point is given by $\lambda(z_0)=R'(z_0)$. Analogously, the multiplier of a periodic point is given by $\lambda(z_0)=(R^p)'(z_0)=R'(z_0)\cdot R'(R(z))\cdot \ldots \cdot R'(R^{p-1}(z))$. A fixed or periodic point $z_0$ is called \textit{attracting} if $|\lambda(z_0)|<1$ (\textit{superattracting} if $|\lambda(z_0)|=0$), \textit{repelling} if $|\lambda(z_0)|>1$, and  \textit{indifferent} if $|\lambda(z_0)|=1$. An indifferent point is called parabolic if $\lambda(z_0)=e^{2\pi i p/q}$, where $p,q\in\N$. All attracting and parabolic fixed points have a basin of attraction associated to them which consists of the set of points which converge to $z_0$ under iteration of $R$. Analogously, every  attracting or parabolic periodic point $z_0$ has a basin of attraction associated to it which consists of the set of points which converge to the cycle $\langle z_0\rangle =\{z_0, R(z_0),\cdots, R^{p-1}(z_0)\}$.

The dynamics of $R$ provides a totally invariant partition of the Riemann sphere. The \textit{Fatou set} $\mathcal{F}(R)$ consists of the set of points $z\in\wcom$ for which the family of iterates of $R(z)$, $\{R(z), R^2(z), \ldots , R^n(z), \ldots \}$,  is normal (or equivalently equicontinuous) in some open neighbourhood of $z$. The Fatou set is open and consists of the set of points for which the dynamics presents stable behaviour. Its complement, the \textit{Julia set} $\mathcal{J}(f)$, is closed and consists of the set of points which present chaotic behaviour. The connected components of the Fatou set, called \textit{Fatou components}, are mapped amongst themselves under iteration of $R$. In view of Sullivan's No-Wandering Theorem (see \cite{Su}), all Fatou components are either periodic or preperiodic.
Every periodic Fatou component either belongs to the basin of attraction of an attracting or parabolic point, or is a simply connected rotation domain (a \textit{Siegel disk}), or is a doubly connected rotation domain (a \textit{Herman ring}).
Moreover, periodic Fatou components of rational maps can be related to \textit{critical points}, i.e.\ points where $R'(z)=0$. Indeed, every cycle of attracting or parabolic Fatou components contains, at least, a critical point. On the other hand, Siegel disks and Herman rings have critical points whose orbits accumulate on their boundaries.

Along the paper we use these concepts to draw both dynamical and parameter planes. \textit{Dynamical planes} show the dynamics of points in a given range. They are drawn by creating a grid of points in the specified range. Afterwards, the point is iterated up to 150 times. If the point converges to a root (the distance to the root is smaller than $10^{-4}$), the iteration stops and we use a scaling from red (fast convergence), to yellow, to green, to blue and to grey (slow convergence) to draw the point. If after 150 iterates the point has not converged to a root, the point is plotted in black.

When a family depends on parameters, it also makes sense to draw parameter planes. \textit{Parameter planes} describe the possible dynamics of a map depending on the parameter.
Since all Fatou components of rational maps are related to critical points, in order to draw parameter planes it is enough to study the orbits of the critical points, the \textit{critical orbits}. In this paper we restrict our drawings to families which only have one \textit{free} critical orbit, i.e.\ an orbit of a critical point other than the superattracting fixed points of the method,  up to symmetry (compare \S~\ref{simetria}). Therefore, in order to draw a parameter plane we restrict to a grid of points in a given range of parameters and iterate a free critical point.
 If the orbit of the critical point converges to a root (the distance to the root is smaller than $10^{-4}$), we conclude that there can be no stable behaviour other than convergence to the roots. In that case the iteration stops and we use a scaling from red (fast convergence), to yellow, to green, to blue and to grey (slow convergence) to draw the point. If after 150 iterates the point has not converged to a root, the point is plotted in black. These points correspond to parameters for which there might be other stable behaviour than convergence to the roots.

\section{Symmetries of operators} \label{simetria}

Symmetries in the dynamical planes of \textit{Newton-like} root-finding algorithms applied on polynomials $z^d-c$ have been observed for different families (see e.g.\ \cite{CCV-gradon},  and \cite{CEGJ}). In this section we prove the need for this symmetry to appear and use it to obtain the operator that appears by applying \textit{Newton-like} root-finding algorithms on degree 2 polynomials.
We start studying the families of maps presenting these symmetry properties
and we show that these maps which can be used to write the expressions of \textit{Newton-like} root-finding algorithms. First, we introduce the concepts of $\lambda^d$-odd and $\lambda^d$-even.

\begin{definition}\label{def:odd}
	Let $d\geq 2$. We say that a map $f:\wcom\rightarrow \wcom$ is $\lambda^d$-\textit{odd} if $ f(\lambda z)= \lambda f( z)$ for all $\lambda\in \C$ such that $\lambda^d=1$ and all $z\in\C$. Respectively, we say that $f$ is $\lambda^d$-\textit{even} 	if $f(\lambda z)= f( z)$.
\end{definition}

Notice that $\lambda^2$-odd and $\lambda^2$-even maps are, precisely, odd and even maps (the only second roots of the unity are 1 and -1). The dynamics of a $\lambda^d$-odd map is symmetric with respect to rotation by $d$th roots of the unity. Indeed, $\lambda^d$-odd maps are conjugated with themselves by multiplication with a $d$th-root of the unity: $f(z)= \lambda ^{-1}f(\lambda z)$. We refer to this property as the symmetry of $\lambda^d$-odd maps. This symmetry is relevant since it allows us to decrease the degrees of freedom of $\lambda^d$-odd maps. It is easy to show that if $\kappa\neq 0,\infty$ is a critical point of a $\lambda^d$-odd map $f$, then the points $\lambda^d \kappa$, where $\lambda^d=1$, are also critical points of $f$. Moreover, their dynamics is tied by the symmetry, so it is enough to control the dynamics of a single critical orbit to know the asymptotic behaviour of $d$ different critical points. The next proposition shows different ways in which we can use the polynomials $p(z)=z^d-c$, $c\in\C\setminus \{0\}$, to obtain $\lambda^d$-odd maps.

\begin{theorem}\label{prop:symn}
	Let $p(z)=z^d-c$, where $d\geq 2$, $c\in\com\setminus\{0\}$. Let $g,h:\WC\rightarrow\WC$ be $\lambda^d$-odd maps and let $H:\WC\rightarrow\WC$ be a $\lambda^d$-even map. Then, the following maps are $\lambda^d$-odd:
	
	\begin{enumerate}[i)]
		\item (the identity) $f(z)=z$ ;
		\item (the linear combination) $f(z)=a\cdot g(z)+b\cdot h(z)$, where $a,b\in\C$;
		\item (the composition) $f(z)=g(h(z))$;
		\item $\displaystyle f(z)=\frac{p(g(z))}{p'(h(z))}$;
		\item $f(z)=\displaystyle h(z)\cdot H(z)$. In particular, $H(z)$ can be obtained as follows:
				\begin{enumerate}[v.1)]
	\item $H(z)=\prod_{i=1}^k \frac{\displaystyle p^{(n_i)}\left(g_i(z)\right)}{\displaystyle p^{(m_i)}\left(h_i(z)\right)}$ is $\lambda^d$-even, where $p^{(n)}$ denotes the $n$th derivative of $p$, $k\geq 1$, $0\leq n_i, m_i\leq d$, $\sum_{i=1}^k m_i=\sum_{i=1}^k n_i$, and $g_i,h_i:\WC\rightarrow\WC$ are $\lambda^d$-odd maps.	
	\item $H(z)=\frac{\displaystyle \sum_{i=1}^k a_i \cdot p^{(n)}\left(g_i(z)\right) }{\displaystyle \sum_{i=1}^\ell b_i\cdot p^{(n)}\left(h_i(z) \right)}$, where $n\geq 0$, $p^{(n)}$ denotes the $n$th derivative of $p$, $k,\ell\geq 1$,   $a_i,b_i\in\C\setminus\{0\}$, and $g_i,h_i:\WC\rightarrow\WC$ are $\lambda^d$-odd maps.
	\item $H(z)= a_1\cdot H_1(z)+a_2\cdot H_2(z)$, where $H_1, H_2:\WC\rightarrow\WC$ are $\lambda^d$-even and $a_1,a_2\in\C$.
	\item $H(z)= H_1(z)/H_2(z)$, where $H_1,H_2:\WC\rightarrow\WC$ are $\lambda^d$-even.
		\end{enumerate}
			\end{enumerate}
\end{theorem}

\begin{proof}
	The cases  \textit{i)}, \textit{ii)}, and \textit{iii)} are straightforward.  We prove now case \textit{iv)}. Let  $f(z)=p(g(z))/p'(h(z))$. Using that $\lambda^d=1$ and $\lambda^{d-1}=\lambda^{-1}$, we have
	
	$$ f(\lambda z)=\frac{p(g(\lambda  z))}{p'(h(\lambda z))}=\frac{p(\lambda g(  z))}{p'(\lambda h(z))}= \frac{(\lambda g(z))^d-c}{d (\lambda h(z))^{d-1}}= \frac{\lambda^d g(z)^d-c}{d \cdot \lambda^{d-1} h(z)^{d-1}}=\frac{ g(z)^d-c}{d \cdot \lambda^{-1} h(z)^{d-1}}=\lambda \frac{ g(z)^d-c}{d\cdot   h(z)^{d-1}}=\lambda f(z).$$
	
	For \textit{v)}, the fact that $f(z)=\displaystyle h(z)\cdot H(z)$ is $\lambda^d$-odd follows directly from the definition of  $\lambda^d$-odd and $\lambda^d$-even.
	 We have to see that the proposed functions $H(z)$  are $\lambda^d$-even. We start with  
	 \[
	 \prod_{i=1}^k \frac{p^{(n_i)}\left(g_i(z)\right)}{p^{(m_i)}\left(h_i(z)\right)}.
	 \]
	 Notice that if $n_i=0$ (or $m_i=0$) then 
	 \[p^{(0)}\left(\lambda g_i(z)\right)= \left(\lambda g_i(z)\right)^d-c= p^{(0)}(g_i(z)). 
	 \]	 
	 If $0<n_i\leq d$, then 
	 \[
	 p^{(n_i)}\left(\lambda g_i(z)\right)=\frac{d!}{(d-n_i)!}(\lambda g_i(z))^{d-n_i}= \lambda^{-n_i}p^{(n_i)}\left(g_i(z)\right).
	 \]
	  Analogously, if $0<m_i\leq d$, then 
	  \[p^{(m_i)}\left(\lambda g_i(z)\right)= \lambda^{-m_i}p^{(m_i)}\left(g_i(z)\right).
	  \]	 
	 
	 The same equalities work if we replace $g_i$ by $h_i$. 
	 Then, the fact that $H$ is $\lambda^d$-even follows then easily from  $\sum_{i=1}^k m_i=\sum_{i=1}^k n_i$.
	
	Similarly, the fact that the function $H$ in \textit{v.2)} is $\lambda^d$-even holds easily using that $p^{(n)}\left(\lambda g_i(z)\right)= \lambda^{-n}p^{(n)}\left(g_i(z)\right).$ 
	The cases \textit{v.3)} and \textit{v.4)} follow directly from the definition of $\lambda^d$-even.
\end{proof}

\begin{remark} 
An important property of the maps constructed in Theorem~\ref{prop:symn} is that they can be used to obtain new Newton-like algorithms.
\end{remark}

Another important feature of the maps constructed in the previous Theorem is that they generally have $z=\infty$ as a simple fixed point: $f$ is a rational map and the degree of its numerator equals the degree of its denominator plus 1.  In the next proposition we describe how this property is preserved under the different constructions in Theorem~\ref{prop:symn}. The proof is straightforward by computing the degree of the numerator and the denominator of the resulting maps.

\begin{proposition}\label{prop:infinityfixed} Let $p(z)=z^d-c$, where $d\geq 2$, $c\in\com\setminus\{0\}$. Let $g,h:\WC\rightarrow\WC$ be two maps that have $z=\infty$ as simple fixed point. Then, the maps $f(z)=g(h(z))$ and $\displaystyle f(z)=\frac{p(g(z))}{p'(h(z))}$ have $z=\infty$ as simple fixed point. Moreover:
	
	\begin{enumerate}[i)]
			
		\item 	The map $f(z)=a\cdot g(z)+b\cdot h(z)$, where $a,b\in\C$, has $z=\infty$ as simple fixed point unless the linear combinations cancel out the highest order term of the numerator.
		
		\item The map $f(z)=\displaystyle h(z)\cdot H(z)$, where $H:\WC\rightarrow\WC$, has $z=\infty$ as simple fixed point provided that $H(z)$ is a rational map whose numerator and the denominator have the same degree.
	\end{enumerate}
	\end{proposition}

\begin{remark}\label{rem:infty}
	 We can analyse whether the different options for $H$ in Theorem~\ref{prop:symn} \textit{v)} satisfy the conditions of Proposition~\ref{prop:infinityfixed} \textit{ii)}. The numerator and denominator of the map $H$ obtained from \textit{v.1)} have the same degree. The map obtained from \textit{v.2)} also satisfies the property provided that the lineal combinations do not decrease the degree at the numerator or the denominator. Analogously, the map $H$ in \textit{v.3)} satisfies the property provided that $H_1$ and $H_2$ also do and there are no simplifications of the higher order term at the numerator. Finally, the numerator and denominator of the map $H$ in \textit{v.4)} have the same degree provided that $H_1$ and $H_2$ also do.
\end{remark}

Since the base map used in Theorem~\ref{prop:symn}, $f(z)=z$ (the identity map), has $z=\infty$ as simple fixed point, we can conclude that the maps obtained fromTheorem~\ref{prop:symn} have $z=\infty$ as simple fixed point provided that the linear combinations given by \textit{ii)}, \textit{v.2)}, and \textit{v.3)} keep the maximal degree. Otherwise, the degree of $z=\infty$ as a fixed point may increase (it becomes a superattracting fixed point) or decrease ($z=\infty$ is no longer a fixed point).

As it has been stated,  the different steps in Theorem~\ref{prop:symn} are relevant for root-finding algorithms since they can be used to obtain many different methods. For instance:

\begin{itemize}
	\item Newton's method $N_p(z)=z-p(z)/p'(z)$ can be obtained from \textit{i)}, \textit{ii)}, and \textit{iv)};
	\item Traub's method $T_p(z)=N_p(z)-p\left(N_p(z)\right)/p'(z)$ can be obtained from \textit{i)}, \textit{ii)}, and \textit{iv)};
	\item Chebyshev's method $C_p(z)=z-\left(1+\frac{1}{2}L(z)\right)p(z)/p'(z)$, where $L(z)=p(z)p''(z)/p'(z)^2$, can be obtained from \textit{i)}, \textit{ii)}, \textit{iv)}, and \textit{v.1)}.
	\item Halley's method $H_p(z)=z-L(z)p(z)/p'(z)$, where $L(z)=\left(1-\frac{p(z)\cdot p''(z)}{2p'(z)^2}\right)^{-1}$, can be obtained from \textit{i)}, \textit{ii)}, \textit{iv)}, and \textit{v)} ($L(z)$ is obtained combining \textit{v.1)}, \textit{v.3)}, and \textit{v.4)}).
	\item Jarratt's method (see \eqref{ecJarrat}) can be obtained from \textit{i)}, \textit{ii)}, \textit{iii)}, \textit{iv)}, and \textit{v.2)}.
	
\end{itemize}

 Let $f$ be a map which is obtained by applying a root finding algorithm that can be built using Theorem~\ref{prop:symn} with $p(z)=z^2-c$. Then, $f$ is odd (by definition $\lambda^2$-odd maps are odd) and $f(\pm\sqrt{c})=\pm\sqrt{c}$ (the roots need to be attracting fixed points of $f$). Following Proposition~\ref{prop:infinityfixed}, we may also assume that $f(\infty)=\infty$. By taking a conjugation $\tau$ which sends $\sqrt{c}$ to $\infty$, $-\sqrt{c}$ to 0, and $\infty$ to $1$, we obtain a new map $\tilde{f}$ with the same dynamics (we just move the roots to $0$ and $\infty$). In the next proposition we describe some properties of $\tilde{f}$.

\begin{proposition}\label{prop:quadraticsym}
	Let $c\in\C\setminus\{0\}$ and let $f:\WC\rightarrow\WC$ be such that $f(\sqrt{c})=\sqrt{c}$, $f(-\sqrt{c})=-\sqrt{c}$, $f(\infty)=\infty$, and $f(-z)=-f(z)$ for all $z\in\C$. Let $\tau(z)=(z+\sqrt{c})/(z-\sqrt{c})$. Then the map $\tilde{f}(z)=\tau\circ f\circ \tau^{-1}(z)$ satisfies  $\tilde{f}(1)=1$ and   $\iota \circ \tilde{f}\circ \iota^{-1}(z)=\tilde{f}(z)$, where $\iota(z)=1/z$.
	
\end{proposition}
\begin{proof}
	Notice that $\tau^{-1}(z)=\sqrt{c}(z+1)/(z-1)$ and $\iota(z)=\iota^{-1}(z)$. Let $\varsigma(z)=-z$. We have
	
	$$\varsigma\circ\tau^{-1}\circ \iota (z)=-\sqrt{c}\frac{1/z+1}{1/z-1}=-\sqrt{c}\frac{1+z}{1-z}=\sqrt{c}\frac{1+z}{z-1}=\tau^{-1}(z);$$
	$$\iota\circ\tau \circ\varsigma(z)=\frac{1}{\displaystyle\frac{-z+\sqrt{c}}{-z-\sqrt{c}}}=\frac{-z-\sqrt{c}}{-z+\sqrt{c}}=\frac{z+\sqrt{c}}{z-\sqrt{c}}=\tau(z).$$
	Using that $f(z)=\varsigma\circ f\circ\varsigma(z)$ we conclude:
	$$\iota\circ \tilde{f}\circ \iota^{-1}(z)=\iota\circ \tau \circ f \circ \tau^{-1}\circ \iota(z)=\iota\circ \tau \circ \varsigma\circ f\circ\varsigma \circ \tau^{-1}\circ \iota(z)=\tau\circ f\circ \tau^{-1}(z)=\tilde{f}.$$
\end{proof}

In Proposition~\ref{prop:quadraticsym} we have proved that $\tilde{f}$ satisfies $\tilde{f}(1)=1$ and   $\iota\circ \tilde{f}\circ \iota(z)=\tilde{f}(z)$, where $\iota(z)=1/z$. In the next theorem we provide a classification of all rational maps satisfying these properties.

\begin{theorem}\label{teorema}
	
	Let $R:\WC\rightarrow \WC$ be a  degree $n$ rational map and let $\iota(z)=1/z$. Then, the following items are equivalent:
	\begin{enumerate}
		\item $R(z)$ satisfies (i) $R(1)=1$ and (ii)  $\iota\circ R\circ \iota^{-1}(z)=R(z)$.
		\item $\displaystyle R(z)=\prod_{i=1}^n\frac{z-r_{i}}{1-r_{i}z}$, where $r_i\in\C$, $i=1,...,n$.
		\item $\displaystyle R(z)=\frac{a_n +a_{n-1} z+...+a_1 z^{n-1}+a_0 z^n}{a_0+a_{1}z+...+a_{n-1} z^{n-1}+a_n z^n}=\frac{P(z)}{\widehat{P}(z)}$, where $a_i\in\C$, $i=0,...,n$.
	\end{enumerate}
\end{theorem}

\begin{proof}
	$1 \Rightarrow 2)$ As $R(z)$ is a degree $n$ rational map, it can be written as:
	\begin{equation}
		R(z)=C\frac{(z-r_{1})\cdot ...\cdot (z-r_{n_{1}})}{(z-s_{1})\cdot ...\cdot
			(z-s_{n_{2}})},  \label{R1}
	\end{equation}
	where at least  $n_{1}=n$ or $n_{2}=n$ and $C\in \mathbb{C}$. Since $\iota^{-1}(z)=\iota(z)$,
	we have that:
	\[
	\iota\circ R \circ \iota^{-1}(z)=\iota\left(R\left(1/z\right)\right)=\iota\left(C \frac{(\frac{1}{z}-r_{1})\cdot
		...\cdot (1/z-r_{n_{1}})}{(1/z-s_{1})\cdot ...\cdot (1/z-s_{n_{2}})}\right)=\frac{1}{C}\frac{(1/z-s_{1})\cdot ...\cdot (1/z-s_{n_{2}})}{(1/z-r_{1})\cdot ...\cdot (1/z-r_{n_{1}})}=R(z).
	\]
	From this expression and comparing with (\ref{R1}), we have that:
	$$
	\left. \begin{array}{l}
		z=1/s_{i},\text{ }i=1,...,n_{2},\text{ are roots of }R \\
		z=r_{i},\text{ }i=1,...,n_{1},\text{ are roots of }R
	\end{array}
	\right\}
	\Longrightarrow n_{2}=n_{1}=n \text{ and } \frac{1}{s_{i}}=r_{j} \text{ for some } j.
	$$
	Let us write  $\frac{1}{s_{i}}=r_{i},$ $i=1,...,n.$ Then, we have that:
	\begin{equation}
		R(z)=\frac{1}{C}\frac{(1/z-s_{1})\cdot ...\cdot (1/z-s_{n})}{
			(1/z-r_{1})\cdot ...\cdot (1/z-r_{n})}=\frac{1}{C}\frac{
			(1-zs_{1})\cdot ...\cdot (1-z s_{n})}{(1-z r_{1})\cdot ...\cdot (1-z r_{n})}
		\label{R2}
	\end{equation}
	
	Denominators in (\ref{R1}) and (\ref{R2}) are two degree $n$ polynomials $
	p(z)$ and $q(z)$ with the same roots (including multiplicity), then $
	p(z)=Kq(z)$  with $K\in \mathbb{C}\ $ and
	\[
	R(z)=C\frac{(z-r_{1})\cdot ...\cdot (z-r_{n})}{K(1-zr_{1})\cdot ...\cdot
		(1-zr_{n})}.
	\]
	Since $R(1)=1,$ we have that $\frac{C}{K}=1$. We  obtain
	
	\[
	R(z)=\frac{(z-r_{1})\cdot ...\cdot (z-r_{n})}{(1-zr_{1})\cdot ...\cdot
		(1-zr_{n})}.
	\]

	$2\Rightarrow 3)$ For this implication we use the induction method.
	This is obviously true for $n=1$:
	\[
	\frac{(z-a)}{(1-a z)}.
	\]
	
	Let us suppose that it is true for $n-1$:
	
	\[
	\frac{a_{n-1} +a_{n-2} z+...+a_1 z^{n-2}+a_0 z^{n-1}}{a_0+a_{1}z+...+a_{n-1} z^{n-1}}=\prod_{i=1}^{n-1}
	\frac{(z-r_i)}{(1-r_i z)}
	\]
	and we check that it is true for $n$:
\begin{eqnarray*}
	& &\frac{(z-r_1)(z-r_2)...(z-r_{n-1})(z-r_{n})}{(1-r_1 z)(1-r_2 z)...(1-r_{n-1} z)(1-r_{n} z)}=  \\
	&=&\frac{a_{n-1} +a_{n-2} z+...+a_1 z^{n-2}+a_0 z^{n-1}}{a_0+a_{1}z+...+a_{n-2} z^{n-2}+a_{n-1} z^{n-1}}\frac{z-r_n}{1-r_n z} = \\
		&=& \frac{-r_n a_{n-1}+(a_{n-1}-r_n a_{n-2})z+...+(a_1-r_n a_0)z^{n-1}+a_0 z^n}{a_0+(a_1 -r_n a_0)z+...+(a_{n-1}-r_n a_{n-2})z^{n-1}-a_{r-1} r_n z^n}=\\
		&=& \frac{b_n +b_{n-1} z+...+b_1 z^{n-1}+b_0 z^n}{b_0+b_{1}z+...+b_{n-1} z^{n-1}+b_n z^n}.
	\end{eqnarray*}
	
	$3\Rightarrow 1)$ Any rational function that satisfies $3)$ also satisfies $1)$. Indeed, it verifies $R(1)=1$ and
	\begin{eqnarray*}
		\iota\circ R \circ \iota^{-1}(z)&=& \iota\left(R\left(\frac{1}{z}\right)\right)=\iota\left(\frac{a_n+a_{n-1}\frac{1}{z}+...a_1 \frac{1}{z^{n-1}}+a_0 \frac{1}{z^n}}
		{a_0+a_1 \frac{1}{z}+...+a_{n-1} \frac{1}{z^{n-1}}+a_n \frac{1}{z^n}}\right)=\\
		&=& \iota\left(\frac{a_n z^n+a_{n-1} z^{n-1}+...+a_1 z+a_0}{a_0 z^n+a_1 z^{n-1}+...+a_{n-1}z+a_n} \right)=\\
		&=& \frac {a_0 z^n+a_1 z^{n-1}+...+a_{n-1}z+a_n}{a_n z^n+a_{n-1} z^{n-1}+...+a_1 z+a_0}=R(z).
	\end{eqnarray*}

\end{proof}

\section{Dynamical study of Newton-like operators} \label{clasificacion}

In \S~\ref{simetria} we have justified that Newton-like  methods applied on complex quadratic polynomials $p(z)=z^2-c$ lead to an  operator $O(z)$ of the form (\ref{operador}), where $n$ denotes the order of convergence to the roots of $p(z)$. Examples of this fact can be found in
   \cite{CCV-gradon+k} \cite{CCTV-familiac},\cite{CV-OstrowskiChun}, \cite{CGMT-pg4},  \cite{CTV-gato}, \cite{ZCT-pg8}, among many other papers.  In \cite{CCV-gradon+k} we study this family of operators when the polynomials of the rational function has $k$ equal roots. In this paper we delve deeper into the  dynamical study of the operator $O(z)$. Following the results from \S~\ref{simetria}, we can write the operator $O(z)$  as:

\begin{equation}\label{op}
    O(z)=z^n \frac{a_k+a_{k-1} z+...+a_{1} z^{k-1}+ z^k}{1+a_{1} z+...+a_{k-1} z^{k-1}+a_k z^k}=z^n \prod_{i=1}^{k}
\frac{(z-r_i)}{(1-r_i z)},
\end{equation}
with $a_k \neq 0$.
Its derivative can be written as:
 \begin{equation}\label{dop}
    O^{\prime}(z)=z^{n-1} \prod_{i=1}^{k}\frac{(z-r_i)}{(1-r_i z)} \left(n+z \sum_{j=1}^{k} \frac{1-r_j ^2}{(z-r_j)(1-r_j z)}   \right).
\end{equation}

The relations among the coefficients  $a_i$  and the roots $r_i$ of the polynomial $P(z)=a_k+a_{k-1} z+...+a_{1} z^{k-1}+ z^k=\prod_{i=1}^{k}(z-r_i)$ are given by:
\begin{eqnarray} \label{relacion_ayr}
	\notag
  a_k &=& (-1)^{k}\prod_{i=1}^{k}(r_i) \\
  a_{k-1}  &=& (-1)^{k-1}(r_1r_2...r_{k-1}+r_1r_3...r_k+...) \\
  \notag
  ...\\
  \notag
  a_2 &=& r_1r_2+r_1r_3+.... \\
  \notag
  a_1 &=& -\sum_{i=1}^{k}r_i.
\end{eqnarray}

Observe that the  subscript of the coefficients indicates the number of roots that appear in the different products.

\begin{remark} \label{rem:3.1} Along this paper we suppose that $z=1$ is not a root of the polynomial $P(z)$ since, if $z = 1$ were a root of  $P(z)$, we could simplify the rational function of the operator. Let us suppose that $r_k=1$. Then,
 \begin{eqnarray}
\notag
   O(z)&=&z^n \frac{a_k+a_{k-1} z+...+a_{1} z^{k-1}+ z^k}{1+a_{1} z+...+a_{k-1} z^{k-1}+a_k z^k}=z^n \prod_{i=1}^{k}
\frac{(z-r_i)}{(1-r_i z)}=\\
\notag
&=& - z^n \prod_{i=1}^{k-1}\frac{(z-r_i)}{(1-r_i z)}= -z^n \frac{b_{k-1}+...+b_{1} z^{k-2}+ z^{k-1}}{1+b_{1} z+...+b_{k-1} z^{k-1}}.
\end{eqnarray}
In this case, the point $z=1$ is not a fixed point of the operator, it forms an orbit of period 2 with $z=-1$ if $k+n$ even or it is a pre-image of $z=-1$ if $k+n$ odd. Therefore, in this paper, we assume that $1+a_1+...+a_k \neq 0$.
\end{remark}

In the following lemma, we provide another relation between the roots and the coefficients of the polynomial.

\begin{lemma}\label{lemma_relacion_ayr}
Given a polynomial $P(z)=a_k+a_{k-1} z+...+a_{1} z^{k-1}+ z^k$  with roots $r_i$, $i=1,...,k$, the following relationship holds:
\[
 \sum_{j=1}^{k} \frac{1+r_j }{1-r_j} = k - 2 \frac{a_1+2a_2+...+k a_k}{1+a_1+a_2+...+a_k}.
\]
\end{lemma}

\begin{proof}
Let us  prove this relation  by induction.
For $k=1$, we have the polynomial $P_1(z)=a_1+z$ with $a_1=-r_1$. Therefore,
\[
\frac{1+r_1}{1-r_1}=1+\frac{2r_1}{1-r_1}= 1- 2 \frac{a_1}{1+a_1}.
\]

Now we assume that  for  the degree $k-1$ polynomial
\[
P_{k-1}(z)=a_{k-1}+a_{k-2} z+...+a_{1} z^{k-2}+ z^{k-1}
\]
the relationship
\[
 \sum_{j=1}^{k-1} \frac{1+r_j }{1-r_j} = (k-1) - 2 \frac{a_1+2a_2+...+(k-1) a_{k-1}}{1+a_1+a_2+...+a_{k-1}}
\]

\noindent holds. Finally, we prove the relation for the degree $k$ polynomial $P_k(z)=(z-r_k)P_{k-1}(z)=b_k+b_{k-1} z+...+b_{1} z^{k-1}+ z^k$:

\begin{eqnarray*}
\sum_{j=1}^{k} \frac{1+r_j }{1-r_j} &=& \sum_{j=1}^{k-1} \frac{1+r_j }{1-r_j} +\frac{1+r_{k}}{1-r_{k}}=k -1- 2 \frac{a_1+2a_2+...+(k-1) a_{k-1}}{1+a_1+a_2+...+a_{k-1}}+1+\frac{2r_{k}}{1-r_{k}}= \\
 &=& k-2\frac{(a_1+2a_2+...+(k-1) a_{k-1})(1-r_{k})-r_{k}(1+a_1+a_2+...+a_{k-1})}{(1+a_1+a_2+...+a_{k-1})(1-r_{k})}= \\
 &=& k-2\frac{a_1-r_{k}+2(a_2-r_{k}a_1)+3(a_3-r_{k}a_2)+...+(k-1)(a_{k-1}-r_{k}a_{k-2})- k a_{k-1} r_{k}}{(1-r_1)...(1-r_{k-1})(1-r_{k})}= \\
 &=& k - 2 \frac{b_1+2b_2+...+(k-1) b_{k-1}+k b_{k}}{1+b_1+b_2+...+b_{k}},
\end{eqnarray*}
where the relations (\ref{relacion_ayr}) have been used.
\end{proof}

\begin{lemma}\label{lemma_relacion_ay-r}
Given a polynomial $P(z)=a_k+a_{k-1} z+...+a_{1} z^{k-1}+ z^k$  with roots $r_i$, $i=1,...,k$, the following relationship holds:
\[
 \sum_{j=1}^{k} \frac{1-r_j }{1+r_j} = k - 2 \frac{-a_1+2 a_2-...+(-1)^k k a_k}{1-a_1+a_2+...+(-1)^k  a_k}.
\]
\end{lemma}

\begin{proof}
The proof is similar to the previous lemma.
\end{proof}

We continue by studying some properties of the Newton-like operators. We  focus on fixed points. Let $f$ be an operator obtained by applying a root finding algorithm, that can be described by Theorem~\ref{prop:symn}, on  a quadratic polynomial $p(z)=z^2-c$. Then, the points $\pm\sqrt{c}$ are attracting fixed points of $f$. Assume also that $f(\infty)=\infty$.  In Proposition~\ref{prop:quadraticsym}, the operator (\ref{op}) was obtained by applying a conjugacy that sends the fixed points $-\sqrt{c}, \sqrt{c},$ and $\infty$ to 0, $\infty$, and 1, respectively. Therefore, $0$ and $\infty$ are fixed points under (\ref{op}) which correspond to the roots of the polynomial while $1$ is a \textit{strange} fixed point (a fixed point which does not correspond to a root). Moreover, the term $z^n$ in  (\ref{op}) indicates that the method has order of convergence $n$ to the roots. Next statement, whose proof is straightforward, indicates that the points $0$, $1$, and $\infty$ are fixed points of (\ref{op}) (even if the operator does not come from a root finding algorithm).

\begin{proposition}\label{0y1}
The points $z=0$, $z=\infty$ and $z=1$ are fixed points of operator (\ref{op}). Moreover $z=-1$ is also a fixed point of (\ref{op}) if $n+k$ is odd and $z=-1$ is a preimage of $z=1$ if $n+k$ is even.
\end{proposition}

When a Newton-like operator (\ref{op}) comes from a one-parameter family of numerical methods, the coefficients $a_i$ usually have a linear dependence with respect to the parameter of the family, as it can be observed in the following subsections.
In paper \cite{CCV-gradon+k}, we study a family of Newton-like operators where the coefficients have a quadratic dependence on the parameter that produces a duplicity of the information in the parameter plane. This duplicity was avoided by redefining the parameter in order to obtain a linear dependence.

As mentioned before, the point $z=1$ is a strange fixed point of the operator (\ref{op}), so it is important to analyse its dynamical behaviour. Assuming a linear relationship between  the coefficients and the parameter of the family  and using Lemma \ref{lemma_relacion_ayr}, we obtain the regions where the point  $z=1$ is attracting in terms of the coefficients.

\begin{proposition}\label{estabilidad_1}
Let $a_i=A_i+ B_i \alpha$ be the linear relationships between the coefficient $a_i$ of the operator (\ref{op}) and the parameter $\alpha$ of the family of rational operators,  with $A_i, B_i\in \R $ and $\alpha\in \C$. Let us denote:

\begin{eqnarray*}
A&=& n+k+ \sum_{j=1}^{k} (n+k-2j)A_j, \  \  \ 
B= \sum_{j=1}^{k} (n+k-2j)B_j, \\ 
A'&=& 1+ \sum_{j=1}^{k} A_j,   \  \  \
B'= \sum_{j=1}^{k} B_j, \\
c&=&(AB-A'B')/(B^2-B'^2), \  \  \ r=(A'B-AB')/(B^2-B'^2).
\end{eqnarray*}

Then, the fixed point $z=1$ satisfies the following statements:
\begin{enumerate}

  \item For $B^2-B'^2 \neq 0 $,
\begin{enumerate}[i)]
 \item $z=1$ is indifferent on the circle C: $\displaystyle \left| \alpha +c  \right| =\left| r\right| $,
  \item $z=1$ is attracting inside C if $B^2-B'^2 > 0$ and outside C if $B^2-B'^2 < 0$. 

  \item $z=1$ is repelling outside C if $B^2-B'^2 > 0$ and inside C if $B^2-B'^2 < 0$.
\end{enumerate}

  \item For $B^2-B'^2 =0$:
  \begin{enumerate}
      \item If $B=B'\neq 0$, then
      \begin{enumerate}[i)]
       \item $z=1$ is indifferent if $A=A'$,
       \item $z=1$ is attracting if $Re(\alpha)<-\frac{A+A'}{2B}$ and $B(A-A')>0$ or \\
\hspace*{3cm} \  $Re(\alpha)>-\frac{A+A'}{2B}$ and $B(A-A')<0$. 
\item $z=1$ is repelling in other case.
 \end{enumerate}

\item If $B=-B'\neq 0$, then
      \begin{enumerate}[i)]
       \item $z=1$ is indifferent if $A=-A'$,
       \item $z=1$ is attracting if $Re(\alpha)<\frac{A'-A}{2B}$ and $B(A+A')>0$ or \\
\hspace*{3cm} \  $Re(\alpha)>\frac{A'-A}{2B}$ and $B(A+A')<0$.
      
\item $z=1$ is repelling in other case.
 \end{enumerate}

\item If $B=B'= 0$, then 

 \begin{enumerate}[i)]
       \item $z=1$ is indifferent if $|A|=|A'|$,
       \item $z=1$ is attracting if $|A|<|A'|$,
       
    \item $z=1$ is repelling  $|A|>|A'|$.
 \end{enumerate} 

\end{enumerate}
\end{enumerate}

Moreover, if $ \alpha=-\frac{A}{B} $ the fixed point $z=1$ is superattracting.
\end{proposition}

\begin{proof}
By sustituting the point $z=1$ in (\ref{dop})
and by applying the result obtained in Lemma \ref{lemma_relacion_ayr}, we  have:
\[
\left| O^\prime (1)\right|= \left|n+\sum _{j=1}^{k}\frac{1+r_j}{1-r_j}    \right|=\left|n+k-2 \frac{a_1+2a_2+...+k a_k}{1+a_1+a_2+...+a_k}\right|.
\]

Assuming the relation  $a_i=A_i+ B_i \alpha$, we obtain:
\[
 \left| O^\prime (1)\right|= \left|\frac{(n+k)(1+A_1+B_1 \alpha +...+A_k+B_k \alpha)-2 (A_1+B_1 \alpha+...+k (A_k+B_k \alpha))}{1+A_1+B_1 \alpha+...+A_k+B_k \alpha}\right|
\]
that can be written as:
\[
 \left| O^\prime (1)\right|=\left|\frac{n+k+\sum_{j=1}^{k} (n+k-2j)A_j +\alpha \sum_{j=1}^{k} (n+k-2j)B_j}{1+\sum_{j=1}^{k} A_j +\alpha \sum_{j=1}^{k}B_j}\right|=\left|\frac{A +\alpha B}{ A' +\alpha B'}\right|.
\]

Now, we look for the values of the parameter that make $\left| O^\prime (1)\right|<1.$
As $\alpha$ is a complex parameter, we can write it as $\alpha = p+i q$, $\ p,q\in \R $. Then,
\[
\left|A+B(p+i q)|<|A'+B'(p+i q)\right| \  \   \Rightarrow  \  \   p^2(B^2-B'^2)+2p(AB-A'B')+q^2(B^2-B'^2)<A'^2-A^2.
 \]
For the case $B^2-B'^2 \neq 0$, this expresion can be divided by $B^2-B'^2$ and, after looking for a perfect square, it can be
 written, in the complex plane,  as the equation of the disk:
\[
\left(p+\frac{AB-A'B'}{B^2-B'^2}\right)^2+ q^2<\left(\frac{A'B-AB'}{B^2-B'^2}\right)^2
\]
if $B^2-B'^2>0$, or the disk
\[
\left(p+\frac{AB-A'B'}{B^2-B'^2}\right)^2+ q^2>\left(\frac{A'B-AB'}{B^2-B'^2}\right)^2
\]
if $B^2-B'^2<0$, obtaining the regions where the point $z=1$ is attractive.

The point $z=1$ is indifferent on the circle
\[
\left(p+\frac{AB-A'B'}{B^2-B'^2}\right)^2+ q^2=\left(\frac{A'B-AB'}{B^2-B'^2}\right)^2
\] and the point $z=1$ is repelling for the other regions.

For the case $B^2-B'^2 =0$, the proof of the latter cases is  straightforward without more than substituting the conditions in $\left| O^\prime(1)\right|$.

Finally, from  $O^{\prime}(1)=0$, we obtain that the point  $z = 1$ is superattracting  for $\alpha = -A/B$.
\end{proof}

Moreover, if $z=-1$ is also a fixed point of (\ref{op}), i.e.\ $n+k$ is odd, we have the following result. The proof is similar to the previous proposition from the  substitution of $z=-1$ in (8) and using Lemma \ref{lemma_relacion_ay-r}.

\begin{proposition}\label{estabilidad_-1}
Let $a_i=A_i+ B_i \alpha$ be the lineal relationship between the coefficients $a_i$ of the operator (\ref{op}) and the parameter $\alpha$ of the family of rational operators,  with $A_i, B_i \in \R $, $\alpha\in \C$  and $n+k$ odd. Let us denote:
\begin{eqnarray*}
C&=& n+ k+ \sum_{j=1}^{k} (-1)^j (n+ k-2j)A_j,  \  \  \ D= \sum_{j=1}^{k} (-1)^j (n+ k-2j)B_j,    \\
C'&=& 1+ \sum_{j=1}^{k} (-1)^j  A_j, \  \  \  D'= \sum_{j=1}^{k}(-1)^j B_j,\\
c&=&(CD-C'D')/(D^2-D'^2), \  \  \   r=(C'D-CD')/(D^2-D'^2).
\end{eqnarray*}

Then, the fixed point $z=-1$ satisfies the following statements:
\begin{enumerate}

\item For $D^2-D'^2 \ne 0 $,
\begin{enumerate}[i)]
  \item $z=-1$ is indifferent on the circle S: $\displaystyle \left| \alpha +c  \right| =\left| r\right| $.
   \item  $z=-1$ is attracting inside the circle S if $D^2-D'^2 > 0 $ and outside the circle S if $D^2-D'^2<0 $.
  \item $z=-1$ is repelling outside the circle S if $D^2-D'^2>0 $ and inside the circle S if $D^2-D'^2<0 $.
\end{enumerate}

\item For $D^2-D'^2 =0$, 
\begin{enumerate}[a)]
\item  If $D=D'\neq 0$, then 

\begin{enumerate}[i)]
       \item $z=-1$ is indifferent if $C=C'$.
       \item $z=-1$ is attracting if $Re(\alpha)<-\frac{C+C'}{2D}$ and $D(C-C')>0$ or \\
\hspace*{3.25cm} \  $Re(\alpha)>-\frac{C+C'}{2D}$ and $D(C-C')<0$. 
\item $z=-1$ is repelling in other case.
 \end{enumerate}

\item  If $D=-D'\neq 0$, then 

  \begin{enumerate}[i)]
       \item $z=-1$ is indifferent if $C=-C'$,
       \item $z=-1$ is attracting if $Re(\alpha)<\frac{C'-C}{2D}$ and $D(C+C')>0$ or \\
\hspace*{3.25cm} \  $Re(\alpha)>\frac{C'-C}{2D}$ and $D(C+C')<0$.
      
\item $z=-1$ is repelling in other case.
 \end{enumerate}

\item If $D=-D'= 0$, then

 \begin{enumerate}[i)]
       \item $z=-1$ is indifferent if $|C|=|C'|$,
       \item $z=-1$ is attracting if $|C|<|C'|$,
       
    \item $z=-1$ is repelling  $|C|>|C'|$.

\end{enumerate}

\end{enumerate}
\end{enumerate}

Moreover, if $ \alpha=-\frac{C}{D} $ the fixed point $z=-1$ is superattracting.
\end{proposition}

As was observed in \cite{CCV-gradon+k}, the dynamical behaviour of the rational operators depends on the relationship between the exponent $n$ of $z$ and the degree $k$ of the polynomial in the rational function of the operator. In fact, the numerical behaviour of the method is very bad when $n<k$ and the better numerical behaviour occurs when $n>k+1$.

\section{ Examples of Newton-like operators} \label{casos}

In this section we overview different methods in the literature. For the case $n=k$, we show some results for a particular case of dependence of the parameter. Along the paper we mention examples of numerical methods for which there is exactly one free critical point, modulo symmetry; so, we can plot their parameter planes iterating a single critical point.

\subsection{The case $n>k$}\label{n>k+1}

The  first family of numerical methods that we studied  satisfying these conditions is the Chebyshev-Halley family, defined by (\ref{fcheby}). A more complete study of the dynamical behaviour of this family can be found in \cite{CTV-gato} when this family is applied on quadratic polynomials and in \cite{CCV-gradon}, when it is applied on a family of $n$-degree polynomials.

After applying the corresponding algorithm on  two-degree polynomials $p(z)=z^2-c$ and enforcing the conjugacy map, we obtained the operator:
\begin{equation}
O_{\alpha}(z)=z^3 \frac{z-2(\alpha-1)}{1-2(\alpha-1)z},
\end{equation}
which corresponds to the operator (\ref{op}) for $n=3$ and $k=1$.
The parameter plane can be seen in Figure \ref{gato}.
\begin{figure}[h!!]
    \centering
    \subfigure{
    \begin{tikzpicture}
    \begin{axis}[width=7cm,  axis equal image, scale only axis,  enlargelimits=false, axis on top]
      \addplot graphics[xmin=-1,xmax=5,ymin=-3,ymax=3] {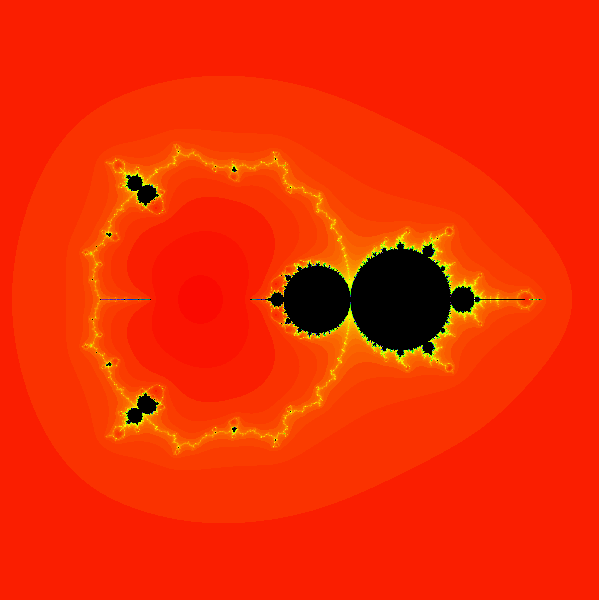};
    \end{axis}
  \end{tikzpicture}}
    \caption{\small{Parameter plane of the Chebyshev-Halley family, $n=3$, $k=1$.}}
    \label{gato}
\end{figure}
As described in \S~\ref{sec:introhol}, the black zones correspond to values of the parameter for which the corresponding dynamical planes have attractors that do not correspond to the roots of the polynomial $p(z)=z^2-c$.

King's family  is another example of a family that leads to an operator  of the form (\ref{op}) for $n = 4$ and $k = 2$, see \cite{CGTVV-king}. The operator of this iterative scheme is given in (\ref{fking}) which leads to:
\begin{equation} \label{opKing}
O_{\beta}(z)=z^4 \frac{5+2\beta+(4+\beta)z+z^2}{1+(4+\beta)z+(5+2\beta)z^2}
\end{equation}
after applying it on a quadratic family of polynomials. The parameter plane is shown in Figure \ref{king-jarrat} (left).

\begin{figure}[h!!]
    \centering
    \subfigure{
    \begin{tikzpicture}
    \begin{axis}[width=7cm,  axis equal image, scale only axis,  enlargelimits=false, axis on top]
      \addplot graphics[xmin=-6,xmax=5,ymin=-5.5,ymax=5.5] {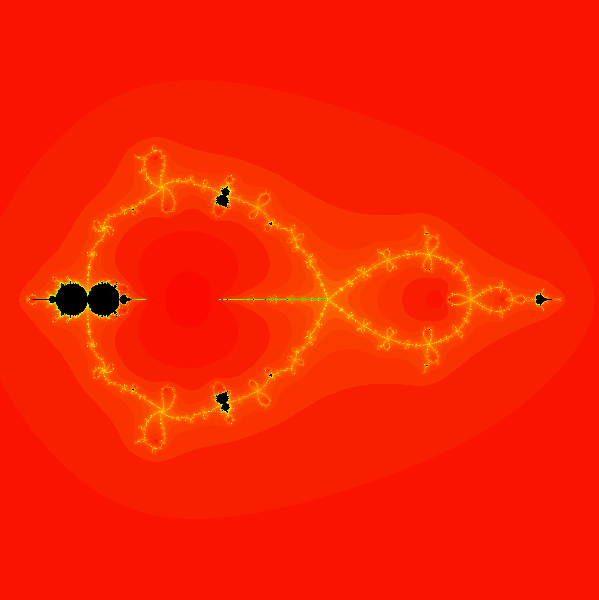};
    \end{axis}
  \end{tikzpicture}}
\subfigure{
	\begin{tikzpicture}
		\begin{axis}[width=7cm,  axis equal image, scale only axis,  enlargelimits=false, axis on top]
			\addplot graphics[xmin=-5,xmax=3,ymin=-4,ymax=4] {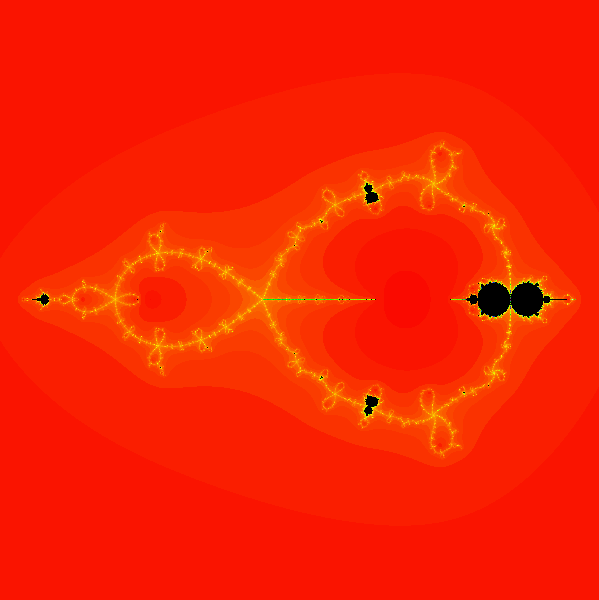};
		\end{axis}
\end{tikzpicture}}
    \caption{\small{Parameter planes of  King's family (left) and a family coming from Jarrat method (right). In both cases $n=4$, $k=2$.}}
    \label{king-jarrat}
\end{figure}

As  mentioned in the introduction, many methods built from Jarratt's method also satisfy the conditions of Theorem~\ref{prop:symn}.  For example, the family of order four studied by Amat et al. in \cite{amat2005} with scheme (\ref{famat}).
Applying it on two-degree polynomials $p(z)=z^2-c$ and enforcing the conjugacy map, the  operator obtained is:
\begin{equation}\label{opJA}
JA_{\beta} \left( z\right) = z^{4}\frac{8\beta-3+(4\beta-6)z-3z^2}{-3+(4\beta-6)z+(8\beta-3)z^2}.
\end{equation}
The parameter plane can be seen in Figure \ref{king-jarrat} (right).

\begin{remark} Let us notice that the operators given in  (\ref{opKing}) and (\ref{opJA}) correspond to the case $n=4$ and $k=2$. In fact, a change $\beta \Rightarrow -\frac{4}{3}\beta -2$ leads to the same operator. So, the parameter planes are the same except for the symmetry induced by the negative sign in the change.
\end{remark}

We find an example of $n=5$ and $k=3$ in the subfamily $S2$ coming from the Ostrowski-Chun methods (see \cite{CV-OstrowskiChun}), whose operator is
\begin{eqnarray} \label{opS2}
OS2_a\left( z\right) &=& z^{5}\frac{14+5a+2(7+2a)z+(6+a)z^2+z^3}{1+(6+a)z+2(7+2a)z^2+(14+5a)z^3}.
\end{eqnarray}
The parameter plane can be seen in Figure \ref{S2}.
\begin{figure}[h!!]
    \centering
    \subfigure{
    \begin{tikzpicture}
    \begin{axis}[width=7cm,  axis equal image, scale only axis,  enlargelimits=false, axis on top]
      \addplot graphics[xmin=-5,xmax=3,ymin=-4,ymax=4] {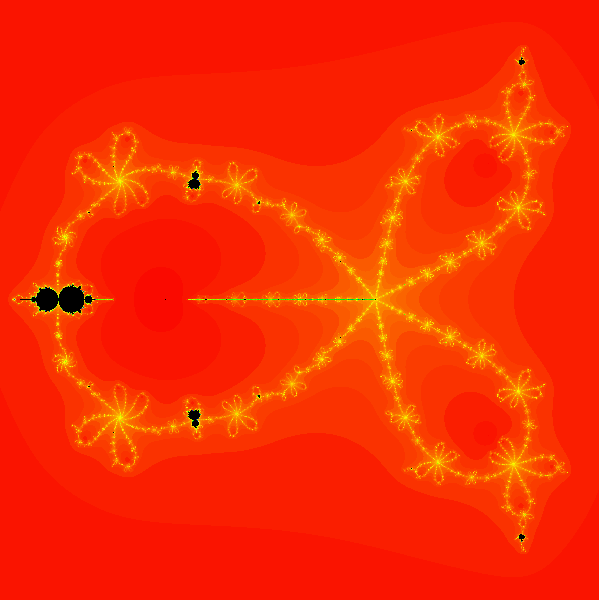};
    \end{axis}
  \end{tikzpicture}}
    \caption{\small{Parameter plane of the subfamily S2, $n=5$, $k=3$}}.
    \label{S2}
\end{figure}

Let us notice that the numerical behaviour of these families is quite good due to the fact that the operators  with presence of strange attractors  correspond to values of the parameter included in the small black zones of the parameter plane.

\subsection{A degenerate case} \label{n=k+1}

An example that coincides with  the degenerate case of Remark \ref{rem:3.1} can be seen in  paper \cite{CV-OstrowskiChun} for the subfamily $S5$ of the Ostrowski-Chun methods. The operator is:
\begin{eqnarray} \label{opS5}
\notag
OS5_{a}\left( z\right) &=& z^4 \frac{-5(7+2a)+(14+5a)z+(14+4a)z^2+(6+a)z^3+z^4}{1+(6+a)z+(14+4a)z^2+(14+5a)z^3-5(7+2a)z^4}=                \\
    &=& -z^{4}\frac{5(7+2a)+(21+5a)z+(7+a)z^2+z^3}{1+(7+a)z+(21+5a)z^2+5(7+2a)z^3}.
\end{eqnarray}
The points $z=1$ and $z=-1$ form an attractive 2-cycle of this operator.

The parameter plane can be seen in Figure \ref{S5} (top). In this figure, green color corresponds to values of the parameter where the free critical point is located in the basin of attraction of the 2-cycle. Black color corresponds to values of the parameter where the orbit of the free critical point does not converge neither to any of the roots of the polynomial nor to the 2-cycle $\{-1,1\}$.

\begin{figure}[h!!]
    \centering
    \subfigure{
    \begin{tikzpicture}
    \begin{axis}[width=7cm,  axis equal image, scale only axis,  enlargelimits=false, axis on top]
      \addplot graphics[xmin=-10.5,xmax=10.5,ymin=-10.5,ymax=10.5] {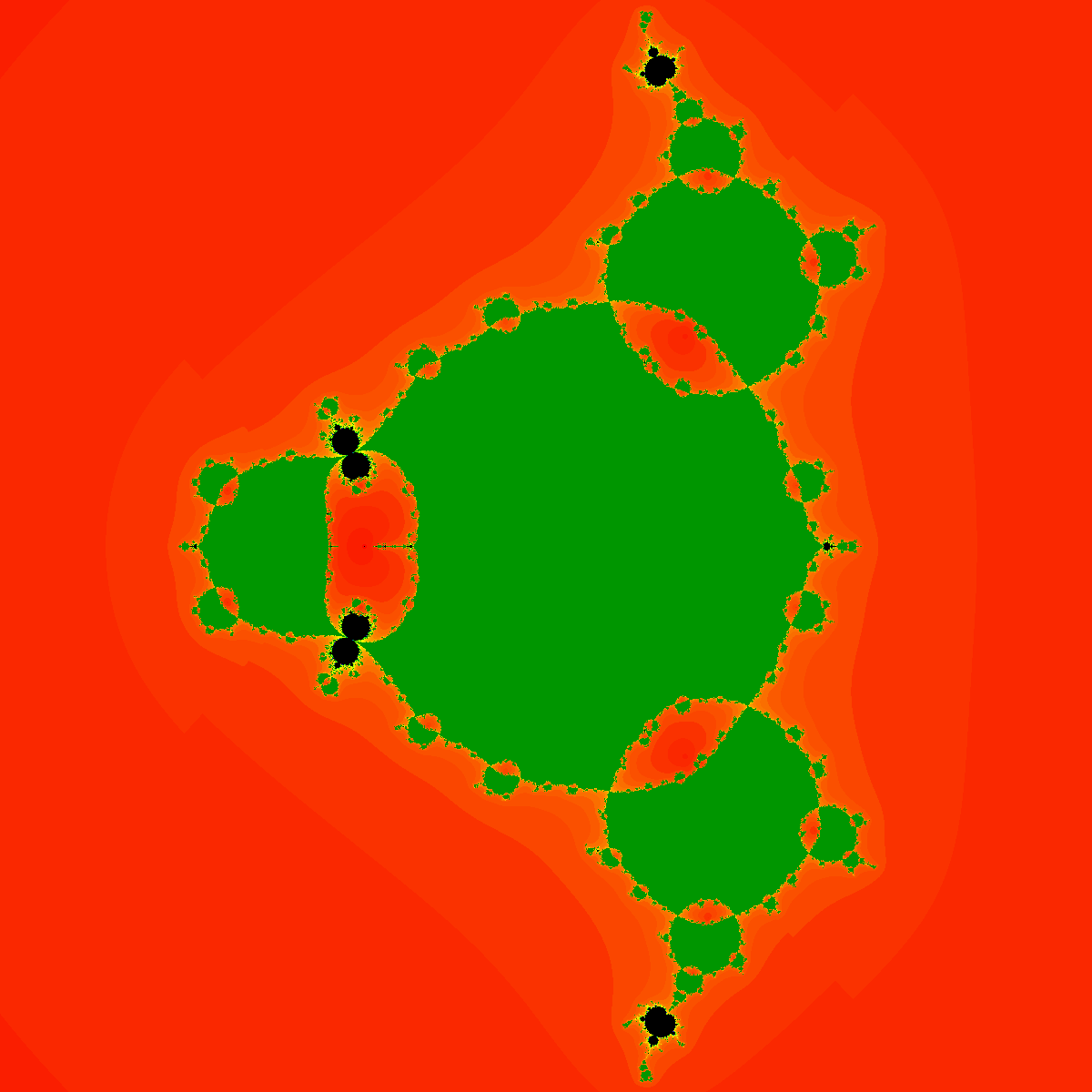};
    \end{axis}
  \end{tikzpicture}}

 \subfigure{

	\begin{tikzpicture}
		\begin{axis}[width=7cm,  axis equal image, scale only axis,  enlargelimits=false, axis on top]
			\addplot graphics[xmin=-9,xmax=5,ymin=-4,ymax=10] {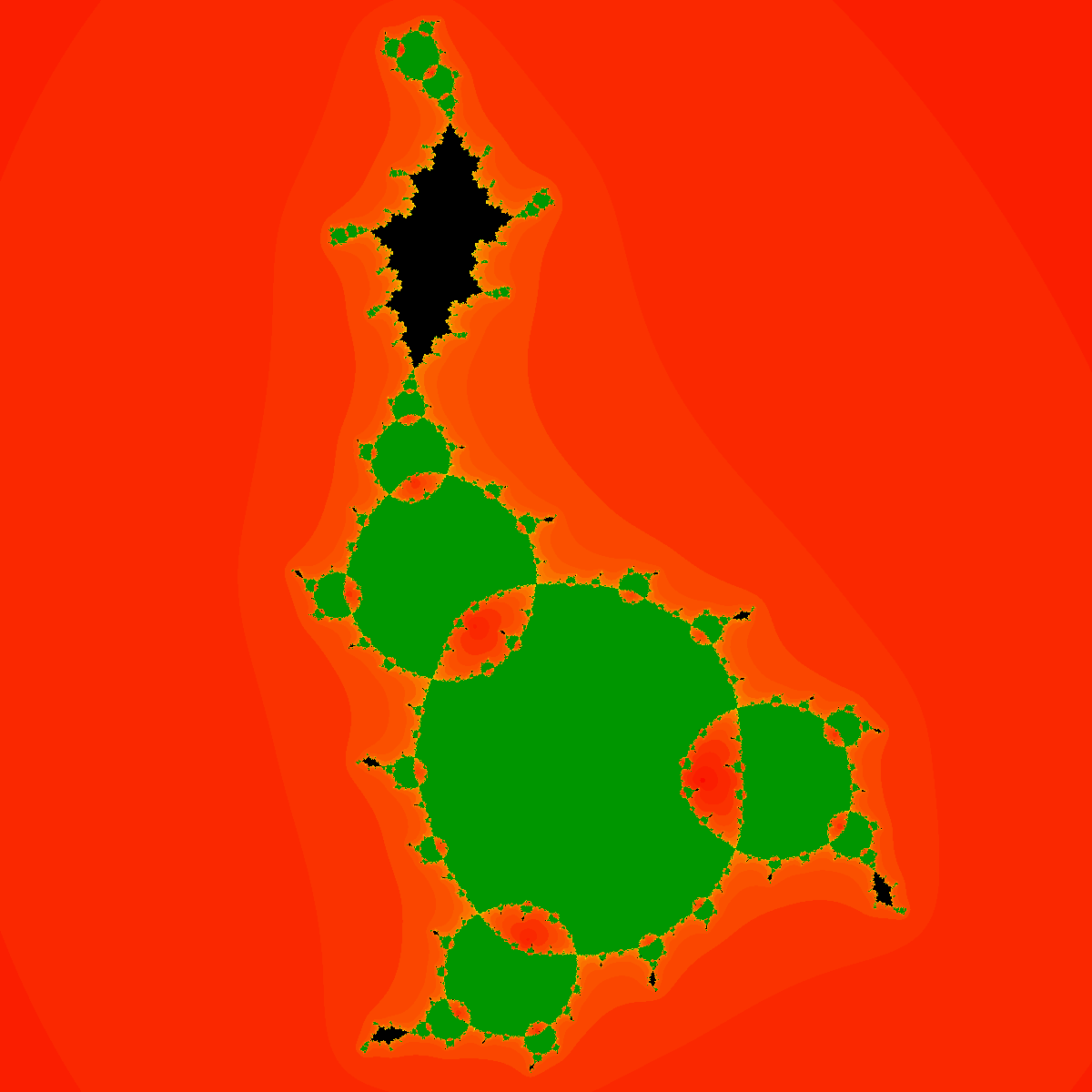};
		\end{axis}
\end{tikzpicture}}
\subfigure{
 \begin{tikzpicture}
    \begin{axis}[width=7cm,  axis equal image, scale only axis,  enlargelimits=false, axis on top]
      \addplot graphics[xmin=-6,xmax=3,ymin=-4.5,ymax=4.5] {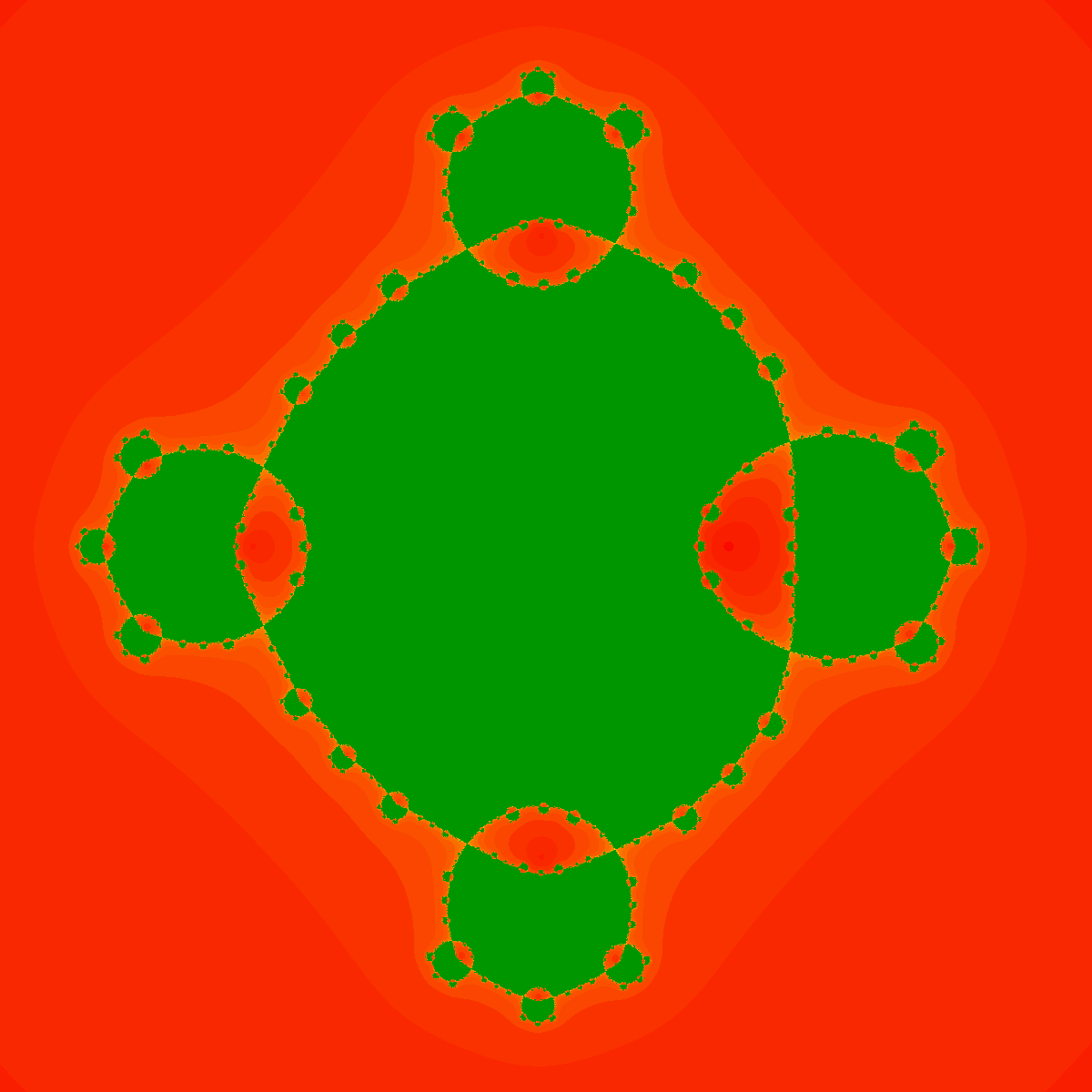};
    \end{axis}
  \end{tikzpicture}}

    \caption{\small{The subfamily S5. The parameter plane (top) and dynamical planes for $a=2-9.3i$ (bottom left) and $a=0$ (bottom right). }}
    \label{S5}
\end{figure}
In Figure \ref{S5} (bottom) we also show two dynamical planes of this family. Colour green indicates the basin of attraction of the 2-cycle $\{-1,1\}$. The left figure corresponds to a value of the parameter where the critical point belongs to the basin of attraction of another attracting cycle.  The right figure corresponds to a value of the parameter for which there is no other stable behaviour  than the basins of attraction of the roots of the polynomial or to the 2-cycle $\{-1,1\}$.

\subsection{The case $n=k$} \label{n=k}

 For $n=k$  the operator is:
\begin{equation}\label{op1}
    O_{\alpha}(z)=z^n \frac{a_n+a_{n-1} z+...+a_{1} z^{n-1}+ z^n}{1+a_{1} z+...+a_{n-1} z^{n-1}+a_n z^n}=
    z^n \prod_{i=1}^{n}\frac{(z-r_i)}{(1-r_i z)}.
\end{equation}
This operator is obtained in many cases (see \cite{CCTV-familiac}, \cite{CV-OstrowskiChun},  \cite{CGMT-pg4}, \cite{ZCT-pg8}, for example).

Next, we consider that all the coefficients of the operator (\ref{op}) are real except $a_n$; that is, $a_i\in \R$ for $1=1,2,...,n-1$ and $a_n=\alpha\in \C$ and we delve deeper into the study of this type of operators.

In this case the point $z=-1$ is a preimage of $z=1$. In the following result we  establish the set of parameters where the fixed point $z=1$ is attracting  for operator (\ref{op1}). We obtain that  $z=1$ is attracting outside a disk with radius related with the coefficients of the polynomial in the numerator of (\ref{op1}). The next proposition is a particular case of Proposition~\ref{estabilidad_1}.

\begin{proposition}\label{prop:estabilida_1n=k}
The fixed point $z=1$ of $O_{\alpha}$ satisfies the following statements:
\begin{enumerate} [i)]
  \item $z=1$ is attracting if $\left| \alpha+A' \right| > \left|A\right| $,
  \item $z=1$ is indifferent if $\left|\alpha+A' \right| = \left|A\right| $,
  \item $z=1$ is repelling if $\left| \alpha+A' \right| < \left|A\right| $,
\end{enumerate}
where $A$ and $A'\neq 0$ are defined as 
\begin{equation*}
   A= 2n+ 2\sum_{j=1}^{n-1} (n-j)a_j,  \  \  \ A'= 1+ \sum_{j=1}^{n-1} a_j.
\end{equation*}

Moreover, if $A=0$ the fixed point $z=1$ is superattracting.
\end{proposition}

\begin{proof}
The proof is straightforward from Proposition~\ref{estabilidad_1} taking into account that, in this case, $a_n=\alpha$, $a_i=A_i$ for $i=1...n-1$, $A_n=0$, $B_i=0$ for $i=1...n-1$, $B_n=1$; so $B=0$ and $B'=1$  and $A= 2n+ 2\sum_{j=1}^{n-1} (n-j)a_j$  and $A'= 1+ \sum_{j=1}^{n-1} a_j$. Then,
\[
|O'(1)|=|\frac{A}{A'+\alpha}|
\]
From the previous expression, it is easy to see that $z = 1$ is superattracting, i.e.\ $O_{\alpha}^{\prime}(1)=0$, for $A = 0$.
\end{proof}

This type of operator is obtained in \cite{CCTV-familiac}, where  the $c$-family of iterative methods is applied on quadratic polynomials. After applying the conjugacy map (see \cite{Blanch84})
the operator is:
\begin{equation} \label{opC}
O_{c}\left( z\right)=z^{3}\frac{2\left( 1-2c\right)+5z+4z^2+z^3 }{1+4z+5z^{2}+2\left( 1-2c\right) z^{3}}
\end{equation}
and the parameter plane can be observed in Figure \ref{ppfamiliaC-mariposa} (left). Notice that the set of parameters for which $z=1$ is attracting corresponds to the unbounded black disk.
\begin{figure}[h!!]
    \centering
    \subfigure{
    \begin{tikzpicture}
    \begin{axis}[width=7cm,  axis equal image, scale only axis,  enlargelimits=false, axis on top]
      \addplot graphics[xmin=-9,xmax=13,ymin=-10,ymax=10] {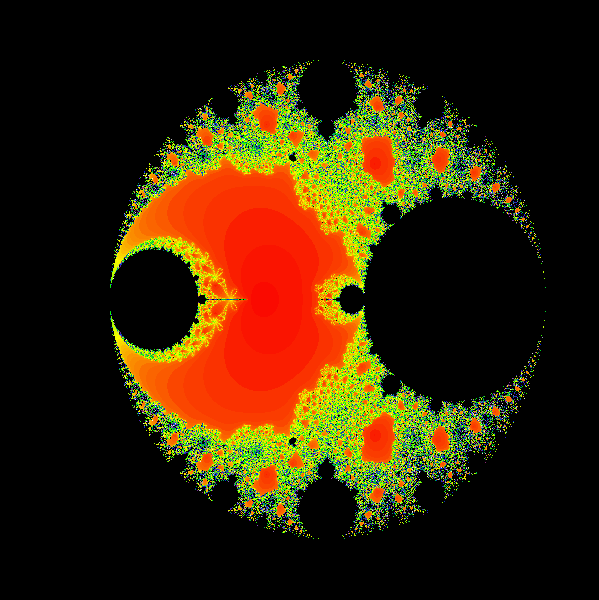};
    \end{axis}
  \end{tikzpicture}}
\subfigure{
	\begin{tikzpicture}
		\begin{axis}[width=7cm,  axis equal image, scale only axis,  enlargelimits=false, axis on top]
			\addplot graphics[xmin=-180,xmax=120,ymin=-150,ymax=150] {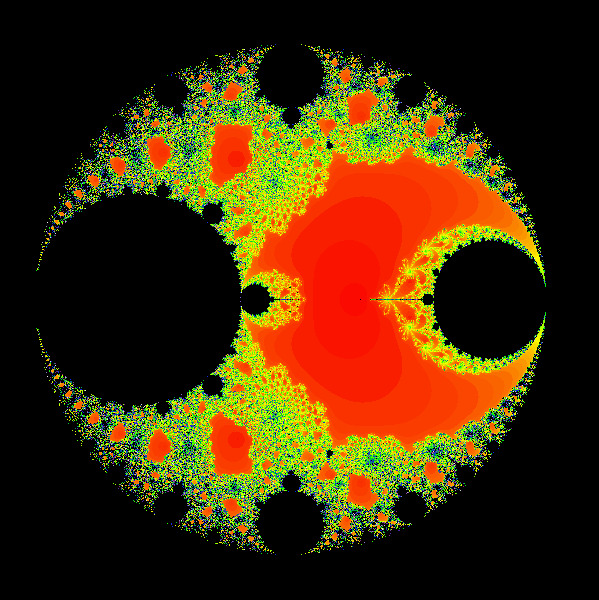};
		\end{axis}
\end{tikzpicture}}

    \caption{\small{Parameter plane of $c-$family for  $n=k=3$ (left) and  the M4 family for $\alpha=\frac{-1+5\beta}{\beta}$ $n=k=4$ (right).}}
    \label{ppfamiliaC-mariposa}
\end{figure}

 A three-step iterative method for solving non-linear equations is studied in \cite{CGMT-pg4}. It is obtained by using the technique of undetermined coefficients and the composition of Newton's scheme with itself, with frozen Jacobian. The authors denote it as M4 and its iterative expression is:
\begin{eqnarray*}
  y_k &=& x_k-F'(x_k)^{-1}F(x_k), \\
 z_k &=& y_k-\frac{1}{\beta}F'(x_k)^{-1}F(y_k), \\
  x_{k+1} &=& z_k- F'(x_k)^{-1}((2-1/\beta -\beta)F(y_k)+\beta F(z_k)).
\end{eqnarray*}

After applying it on quadratic polynomials and enforce the conjugacy map, the operator obtained  is:
\begin{eqnarray} \label{OPmariposa}
O_{\beta}(z)=z^4 \frac{-1+5\beta+14\beta z +14\beta z^2+6\beta z^3+ \beta z^4}{\beta+6\beta z+14\beta z^2+14\beta z^3+(-1+5\beta)z^4} .
\end{eqnarray}
This operator can be transformed into an operator of the type (\ref{op1}) by means of the change of the parameter $\alpha=\frac{-1+5\beta}{\beta}$:
\[
O_{\alpha}(z)=z^4 \frac{\alpha+14  z +14  z^2+6  z^3+   z^4}{1+6  z+14  z^2+14  z^3+\alpha z^4}.
\]
Its parameter plane (see Figure \ref{ppfamiliaC-mariposa} right) is similar to the previous one (Figure~\ref{ppfamiliaC-mariposa} left). Despite they correspond to different values of $n$, the main difference between these parameter planes comes from the different sign of the parameters $c$ and $\alpha$. This leads to a change of positions of the main interior bulb.

Another example of  $n=4$ is obtained from a sub-family of operators coming from the Ostrowski- Chun methods (see \cite{CV-OstrowskiChun}).
Applying this method on quadratic polynomials, and studying one of the subfamilies obtained, the resulting operator is:
\begin{eqnarray*}
OS4_{b}\left( z\right) &=& z^{4}\frac{4b-3-6 z-2 z^{2}+2z^3+z^4}{ 1+2z-2z^{2}-6z^{3}+\left( -3+4b\right) z^{4} }.
\end{eqnarray*}

The parameter plane of this family is shown in Figure \ref{familiaS4} (top). Let us realize that, in this case, the fixed point $z=1$ is superattracting. Therefore, this family has bad initial conditions for every parameter. Since  $z=1$ is superattracting for all values of the parameter, we cannot observe the disks that appears in the parameter planes of Figures \ref{ppfamiliaC-mariposa}. Since $z=1$ is superattracting, $z=1$ does not need to have a critical point (other than itself) in its basin of attraction. 
 As in Figure \ref{S5}, we use the green colour when the critical orbit converges to the basins of attraction of the fixed point $z=1$. 

\begin{figure}[h!!]
    \centering
    \subfigure{
    \begin{tikzpicture}
    \begin{axis}[width=7cm,  axis equal image, scale only axis,  enlargelimits=false, axis on top]
      \addplot graphics[xmin=-6,xmax=8,ymin=-7,ymax=7] {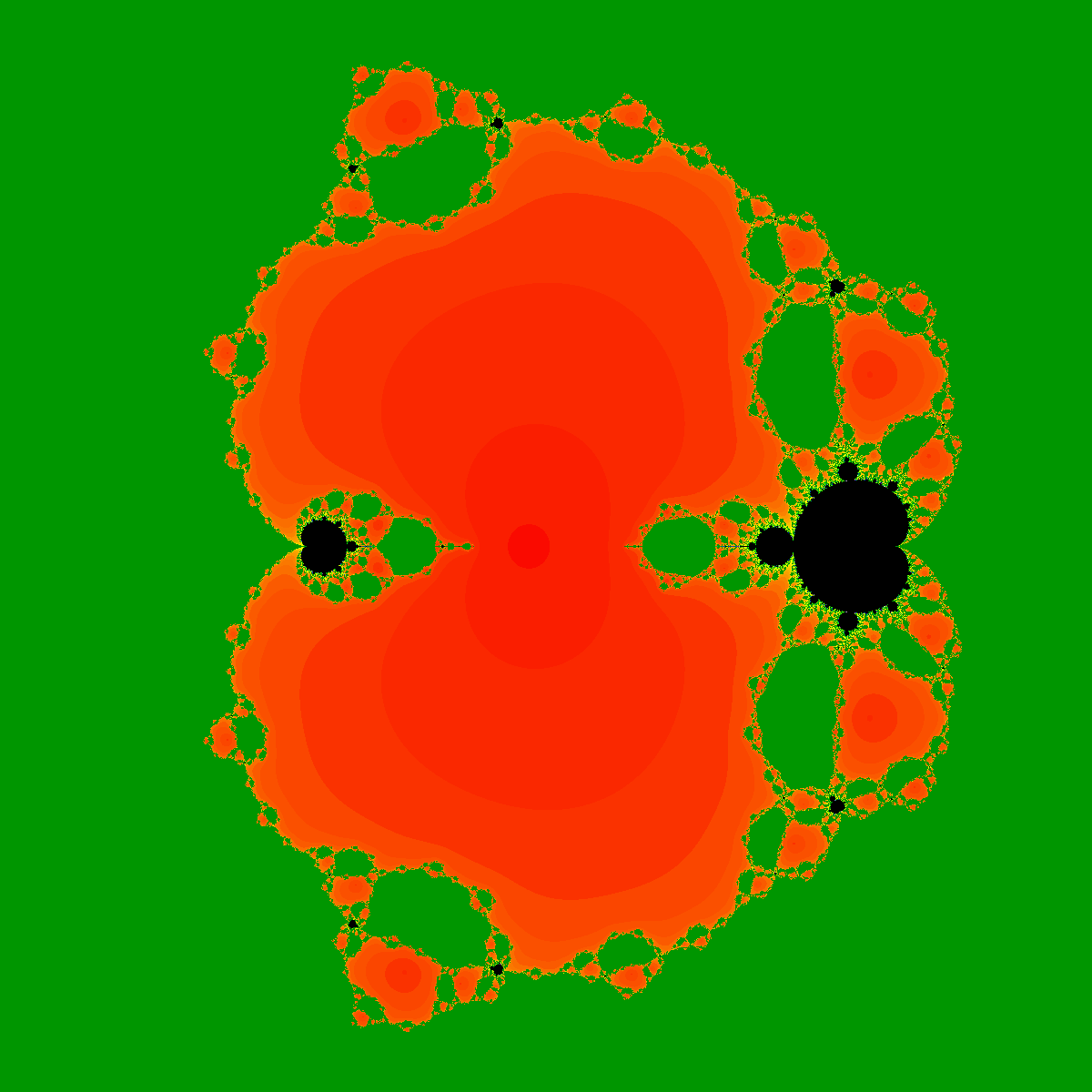};
    \end{axis}
  \end{tikzpicture}}

  \subfigure{
	\begin{tikzpicture}
		\begin{axis}[width=7cm,  axis equal image, scale only axis,  enlargelimits=false, axis on top]
			\addplot graphics[xmin=-4,xmax=3,ymin=-3.5,ymax=3.5] {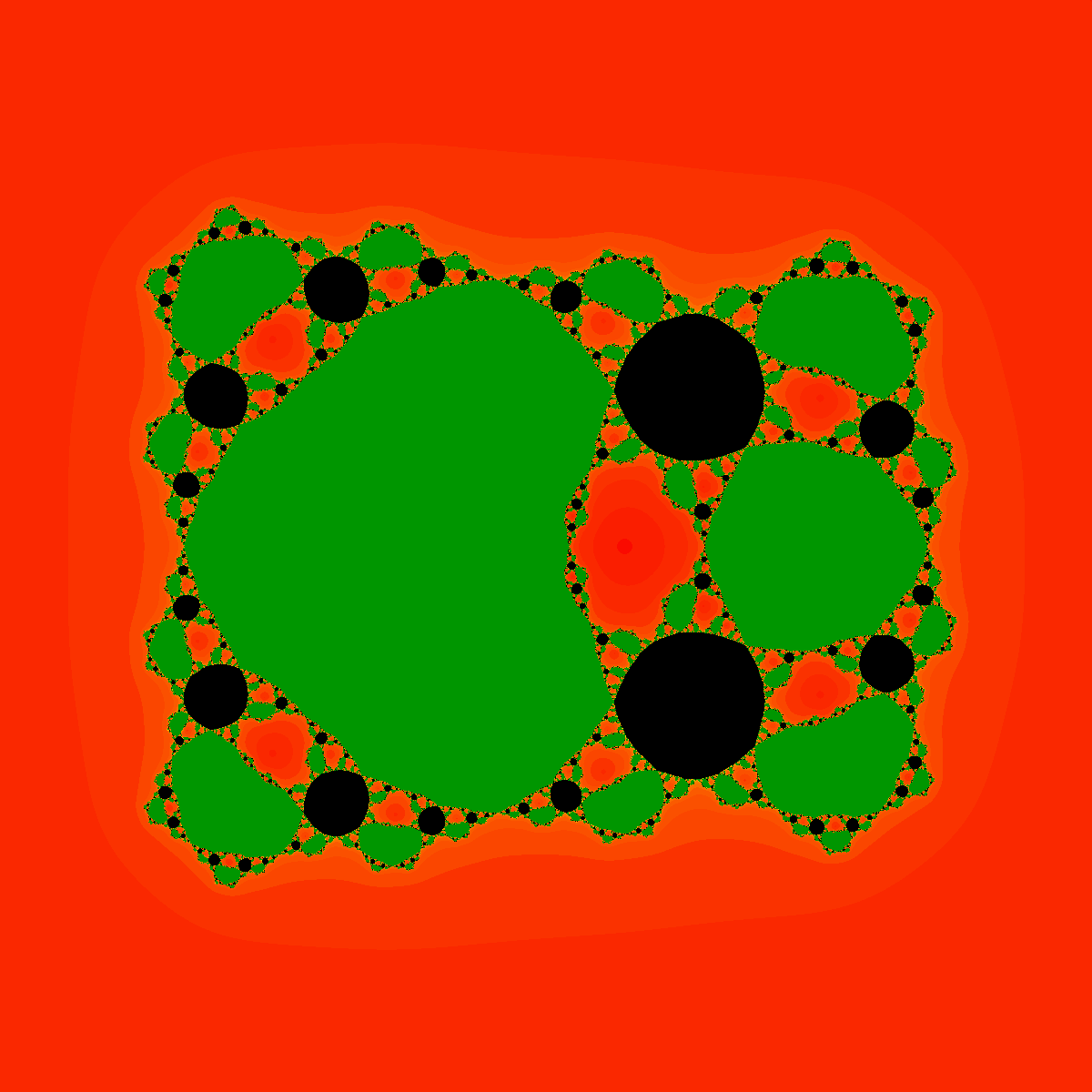};
		\end{axis}
\end{tikzpicture}}
\subfigure{
	\begin{tikzpicture}
		\begin{axis}[width=7cm,  axis equal image, scale only axis,  enlargelimits=false, axis on top]
			\addplot graphics[xmin=-3.5,xmax=2.5,ymin=-3,ymax=3] {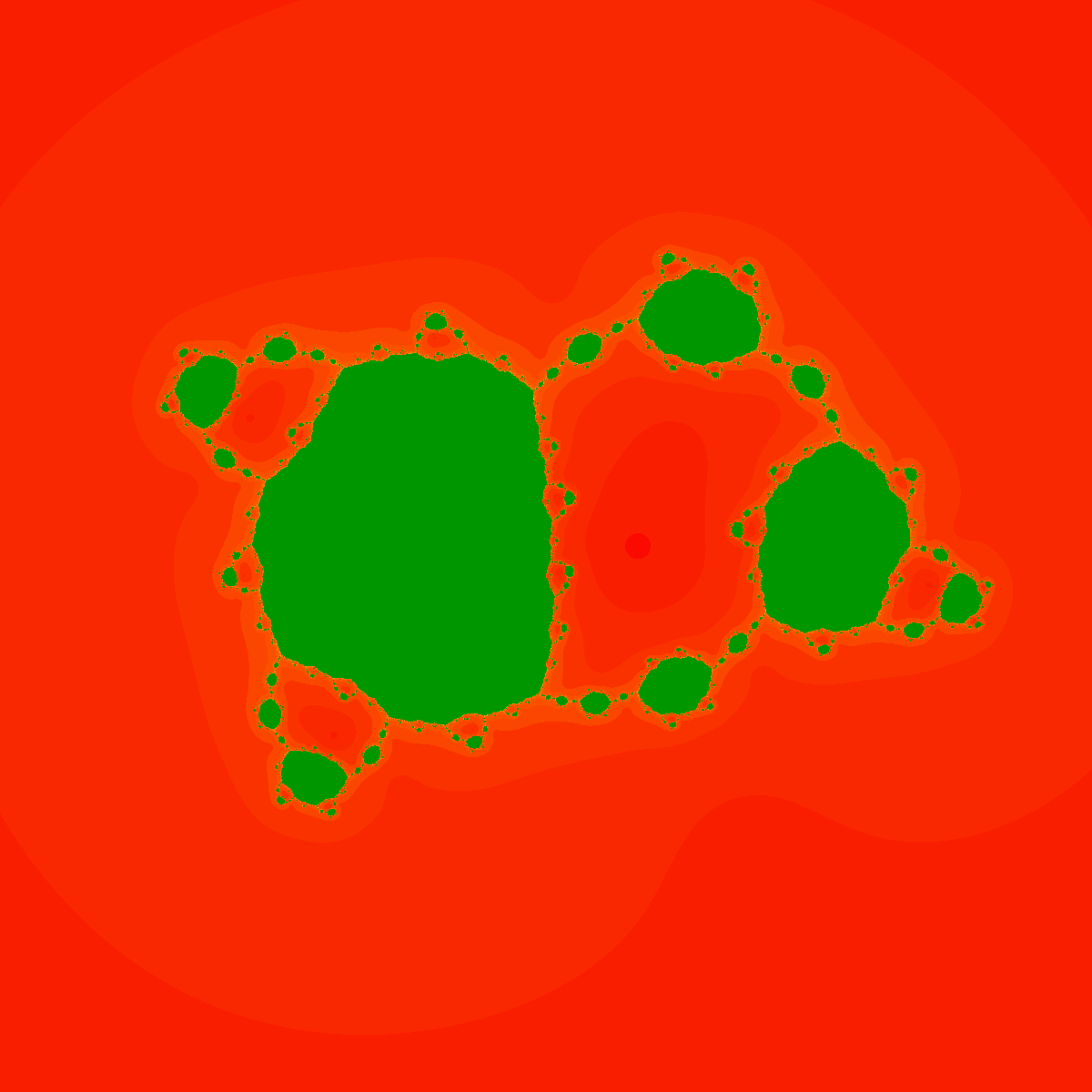};
		\end{axis}
\end{tikzpicture}}
    \caption{\small{The subfamily $S4$ for  $n=4$, $k=4$. The parameter plane (top) and dynamical planes for $b=5$ (bottom left) and $b=1+1i$ (bottom right).}}
    \label{familiaS4}
\end{figure}

In Figure \ref{familiaS4} (bottom) we also show two dynamical planes of this family. Colour green indicates the basin of attraction of  $z=1$. The left figure corresponds to a value of the parameter where the critical point belongs to the basin of attraction of another attracting cycle.  The right figure corresponds to a value of the parameter for which there is not other stable behaviour than the basins of attraction of the roots of the polynomial or $z=1$.

Let us remark that in Figures \ref{ppfamiliaC-mariposa} the region of parameters where $z=1$ is attracting corresponds to the exterior black zone. Similar parameter planes are also obtained in  \cite{CGMT-pg4}, \cite{ZCT-pg8}. Although in each of these papers a different numerical method is studied, the operator obtained after applied it on quadratic polynomials is of the type given in Equation (\ref{op1}). We also want to point out that for Figure \ref{familiaS4} (top) the unbounded region corresponds to a set of parameters for which the free critical points are captured by the basin of attraction of $z=1$.

Another example where the operator is given by Equation (\ref{op}) for $n=k=4$ can be found in the $S3$ family studied in paper \cite{CV-OstrowskiChun}, whose operator is:
\begin{eqnarray} \label{opS3}
OS3_{a}\left( z\right) &=& z^{4}\frac{5(14+5a)^2+(196+76a+9a^2)((14+5a)z+2(7+2a)z^2+(6+a)z^3+z^4)}{(1+(6+a)z+2(7+2a)z^2+(14+5a)z^3)(196+76a+9a^2)+5(14+5a)^2 z^4}.
\end{eqnarray}

Let us notice that, in this case,  there is not a linear relationship between the coefficients of the operator and the parameter of the family. 
The parameter plane can be seen in Figure \ref{familiaS3}.

\begin{figure}[h!!]
    \centering
    \subfigure{
    \begin{tikzpicture}
    \begin{axis}[width=7cm,  axis equal image, scale only axis,  enlargelimits=false, axis on top]
      \addplot graphics[xmin=-6.5,xmax=3.5,ymin=-5,ymax=5] {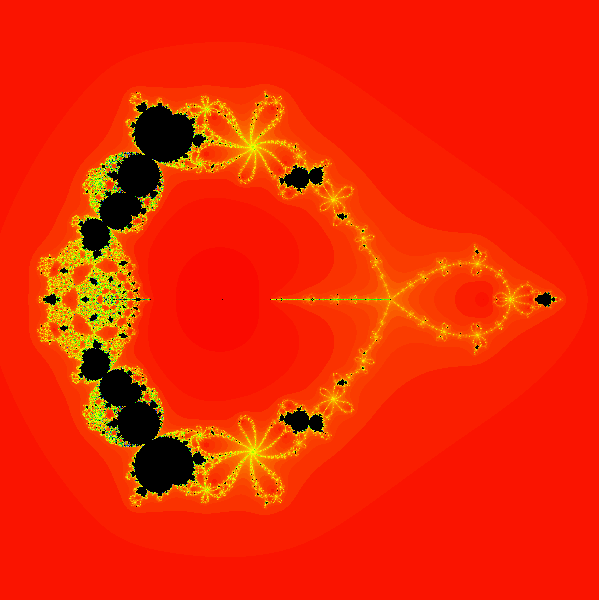};
    \end{axis}
  \end{tikzpicture}}
    \caption{\small{Parameter plane of the subfamily $S3$, $n=4$, $k=4$.}}
    \label{familiaS3}
\end{figure}

\section{Conclusions}

When  Newton-like algorithms are applied on quadratic polynomials, an intrinsic symmetry appears in the operator obtained.

Firstly, we study some symmetry properties of certain families of maps which can be used to obtain the expressions of Newton-like finding algorithms. We show that many of the current numerical methods can be obtained from the functions defined in the Theorem \ref{prop:symn}.

Moreover, the functions introduced in Theorem~\ref{prop:symn} can be used to generate new Newton-type algorithms for solving nonlinear equations. 

We conclude from Theorem \ref{teorema} that the operators obtained for the different families of numerical Newton-like methods have the same  expression, regardless of the method used, when they are applied on quadratic polynomials $p(z)=z^2-c$.

 We carry out a dynamical study for this generic operator. Moreover, except for the degenerate cases,  $z=1$ is always a fixed point of the operator and we study its stability region (Proposition \ref{estabilidad_1}).  We also prove that $z = -1$ is a fixed point when the degree of the operator is odd and we locate its stability region (Proposition \ref{estabilidad_-1}).

Along this paper, we exhibit some parameter planes  coming from different papers of literature where predicted results are observed.

\vspace{1cm}

\textbf{Author' contribution}. These authors contributed equally.

\vspace{0.5cm}
\textbf{Funding} The authors are supported the project UJI-B2019-18. The first and the last authors are also supported by by the grant PGC2018-095896-B-C22. The second author is supported by PID2020-118281GB-C32.
(MCIU/AEI/FEDER/UE).

\vspace{0.5cm}
\textbf{Data availability} Not applicable.

\vspace{0.5cm}
\textbf{Ethical Approval} Not applicable.

\vspace{0.5cm}
\textbf{Acknowledgments} Not applicable.

\vspace{0.5cm}
\textbf{Consent for publication} All authors give their consent for publication.

\vspace{0.5cm}
\textbf{Competing interests} The author declares no competing interests.

\vspace{1cm}
\bibliography{bibliografia}

\begin{thebibliography}{10}

\bibitem{amat2005}
S.~Amat, S.~Busquier, and S.~Plaza.
\newblock Dynamics of the {K}ing and {J}arratt iterations.
\newblock {\em Aequationes Math.}, 69(3):212--223, 2005.

\bibitem{AM-pg5}
I.K. Argyros and A.A. Magre{\~n}\'{a}n.
\newblock On the convergence of an optimal fourth-order family of methods and
  its dynamics.
\newblock {\em Appl. Math. Comput.}, 252:336--346, 2015.

\bibitem{Blanch84}
P.~Blanchard.
\newblock Complex analytic dynamics on the {R}iemann sphere.
\newblock {\em Bull. Amer. Math. Soc. (N.S.)}, 11(1):85--141, 1984.

\bibitem{Blanchard94}
P.~Blanchard.
\newblock {\em The dynamics of {N}ewton's method}, volume~49 of {\em Proc.
  Sympos. Appl. Math.}
\newblock Amer. Math. Soc., Providence, RI, 1994.
\newblock Complex dynamical systems ({C}incinnati, {OH}, 1994).

\bibitem{CCV-gradon+k}
B.~Campos, J.~Canela, A.~Garijo, and P.~Vindel.
\newblock Dynamics of a family of rational operators of arbitrary degree.
\newblock {\em Math. Model. Anal.}, 26(2):188--208, 2021.

\bibitem{CCV-gradon}
B.~Campos, J.~Canela, and P.~Vindel.
\newblock Convergence regions for the {C}hebyshev-{H}alley family.
\newblock {\em Commun. Nonlinear Sci. Numer. Simul.}, 56(3):508--525, 2018.

\bibitem{CCTV-familiac}
B.~Campos, A.~Cordero, J.R. Torregrosa, and P.~Vindel.
\newblock Dynamics of the family of c-iterative methods.
\newblock {\em Int. J. Comput. Math.}, 92(9):1815--1825, 2015.

\bibitem{CV-OstrowskiChun}
B.~Campos and P.~Vindel.
\newblock Dynamics of subfamilies of {O}strowski-{C}hun methods.
\newblock {\em Math. Comput. Simulation}, 181:57--81, 2021.

\bibitem{CEGJ}
J.~Canela, V.~Evdoridou, A.~Garijo, and X.~Jarque.
\newblock On the basins of attraction of a one dimensional family of root
  finding algorithms: From newton to traub.
\newblock Preprint, 2021.

\bibitem{chun}
Ch. Chun.
\newblock Some improvements of {J}arratt's method with sixth-order convergence.
\newblock {\em Appl. Math. Comput.}, 190(2):1432--1437, 2007.

\bibitem{CHUN2012}
Ch. Chun, M.Y. Lee, B.~Neta, and J.~D{\v{z}}uni{\'c}.
\newblock On optimal fourth-order iterative methods free from second derivative
  and their dynamics.
\newblock {\em Appl. Math. Comput.}, 218(11):6427--6438, 2012.

\bibitem{CGTVV-king}
A.~Cordero, J.~Garc\'{\i}a-Maim\'{o}, J.R. Torregrosa, M.~P. Vassileva, and
  P.~Vindel.
\newblock Chaos in {K}ing's iterative family.
\newblock {\em Appl. Math. Lett.}, 26(8):842--848, 2013.

\bibitem{CGMT-pg4}
A.~Cordero, J.M. Guti\'{e}rrez, A.A. Magre{\~n}{\'a}n, and J.R. Torregrosa.
\newblock Stability analysis of a parametric family of iterative methods for
  solving nonlinear models.
\newblock {\em Appl. Math. Comput.}, 285:26--40, 2016.

\bibitem{CTV-gato}
A.~Cordero, J.R. Torregrosa, and P.~Vindel.
\newblock Dynamics of a family of {C}hebyshev-{H}alley type methods.
\newblock {\em Appl. Math. Comput.}, 219(16):8568--8583, 2013.

\bibitem{Gutierrezcubicos}
J.M. Guti\'errez and J.L. Varona.
\newblock Superattracting extraneous fixed points and n-cycles for
  {C}hebyshev's method on cubic polynomials.
\newblock {\em Qual. Theory Dyn. Syst.}, 19(2):54--23, 2020-08.

\bibitem{jarratt}
P.~Jarratt.
\newblock Some efficient fourth order multipoint methods for solving equations.
\newblock {\em BIT.}, 9(2):119--124, 1969-6.

\bibitem{junjua}
M.~Junjua, S.~Akram, N.~Yasmin, and F.~Zafar.
\newblock A new {J}arratt-type fourth-order method for solving system of
  nonlinear equations and applications.
\newblock {\em J. Appl. Math.}, 2015:Art. ID 805278, 14 pp., 2015.

\bibitem{king}
R.F. King.
\newblock A family of fourth order methods for nonlinear equations.
\newblock {\em SIAM J. Numer. Anal.}, 10(5):876--879, 1973-10.

\bibitem{Kneisl}
K.~Kneisl.
\newblock Julia sets for the super-{N}ewton method, {C}auchy's method and
  {H}alley's method.
\newblock {\em Chaos}, 11(2):359--370, 2001-06.

\bibitem{MA-pg2}
A.A. Magre{\~n}{\'a}n and I.K. Argyros.
\newblock On the local convergence and the dynamics of {C}hebyshev-{H}alley
  methods with six and eight order of convergence.
\newblock {\em J. Comput. Appl. Math.}, 298:236--251, 2016.

\bibitem{mcnamee2013numerical}
J.M. McNamee and V.~Pan.
\newblock {\em Numerical Methods for Roots of Polynomials. Part {II}}.
\newblock Newnes, 2013.

\bibitem{Milnor}
J.~Milnor.
\newblock {\em Dynamics in one complex variable}, volume 160 of {\em Annals of
  Mathematics Studies}.
\newblock Princeton University Press, Princeton, NJ, third edition, 2006.

\bibitem{radioCov}
S.~Măruşter.
\newblock Local convergence and radius of convergence for modified newton
  method.
\newblock {\em Annals of West University of Timisoara - Mathematics and
  Computer Science}, 55, 12 2017.

\bibitem{ostrowski2016solution}
A.M. Ostrowski.
\newblock {\em Solution of Equations and Systems of Equations}, volume~9 of
  {\em Pure and Applied Mathematics: A Series of Monographs and Textbooks}.
\newblock Elsevier, 2016.

\bibitem{petkovic}
M.S. Petkovi{\'c}, B.~Neta, L.J. Petkovi{\'c}, and J.~D{\v{z}}uni{\'c}.
\newblock Multipoint methods for solving nonlinear equations: A survey.
\newblock {\em Appl. Math. Comput.}, 226:635--660, 2014.

\bibitem{PETKOVIC2013}
M.~S. Petković, B.~Neta, L.D. Petković, and J.~Džunić.
\newblock {\em Multipoint Methods for Solving Nonlinear Equations}.
\newblock Academic Press, Boston, 2013.

\bibitem{Su}
D.~Sullivan.
\newblock Quasiconformal homeomorphisms and dynamics. {I}. {S}olution of the
  {F}atou-{J}ulia problem on wandering domains.
\newblock {\em Ann. of Math. (2)}, 122(3):401--418, 1985.

\bibitem{traub1982}
J.F. Traub.
\newblock {\em Iterative methods for the solution of equations}, volume 312.
\newblock American Mathematical Soc., 1982.

\bibitem{wang}
X.~Wang, J.~Kou, and Y.~Li.
\newblock Modified {J}arratt method with sixth-order convergence.
\newblock {\em Appl. Math. Lett.}, 22(12):1798--1802, 2009.

\bibitem{ZCT-pg8}
F.~Zafar, A.~Cordero, and J.R. Torregrosa.
\newblock Stability analysis of a family of optimal fourth-order methods for
  multiple roots.
\newblock {\em Numer. Algorithms}, 81(3):947--981, 2019.

\end{thebibliography}
\bibliographystyle{plain}

\end{document}